\documentclass[11pt]{amsart}
\usepackage{mathrsfs}
\usepackage{amscd, amssymb, amsmath, amsthm, amsfonts}
\usepackage{latexsym, graphics, graphicx, psfrag}
\usepackage{color,xcolor,colortbl}
\usepackage[all]{xy}

\newtheorem{thm}{Theorem}[section]
\newtheorem{col}[thm]{Corollary}
\newtheorem{lem}[thm]{Lemma}
\newtheorem{prop}[thm]{Proposition}
\newtheorem{que}[thm]{Question}
\theoremstyle{definition}
\newtheorem{defn}[thm]{Definition}
\theoremstyle{remark}
\newtheorem{rem}[thm]{Remark}
\theoremstyle{example}

\usepackage[linkcolor=red, citecolor=blue]{hyperref}

\title{Virtual Homological Torsion of Closed Hyperbolic $3$-manifolds}

\author[Hongbin~Sun]{Hongbin Sun}
\address{%
    Department of Mathematics\\
    Princeton University\\
    Princeton, NJ 08544, USA}
\email{%
    hongbins@math.princeton.edu}

\date{}

\begin{document}

\begin{abstract} In this paper, we will use Kahn and Markovic's immersed almost totally geodesic surfaces (\cite{KM1}) to construct certain immersed $\pi_1$-injective $2$-complexes in closed hyperbolic $3$-manifolds. Such $2$-complexes are locally almost totally geodesic except along a $1$-dimensional subcomplex. By using Agol 's result that the fundamental groups of closed hyperbolic $3$-manifolds are vitually compact special (\cite{Ag}, \cite{Wi}) and other works on geometric group theory, we will show that any closed hyperbolic $3$-manifold virtually contains any prescribed subgroup in the homological torsion. More precisely, our main result is, for any finite abelian group $A$, and any closed hyperbolic $3$-manifold $M$, $M$ admits a finite cover $N$, such that $A$ is a direct summand of $Tor(H_1(N;\mathbb{Z}))$.
\end{abstract}

\maketitle


\section{Introduction}
\subsection{Background}
In \cite{Lu1}, L\"{u}ck showed that the $L^2$-betti numbers of a CW-complex with residually finite fundamental group can be approximated by the betti numbers of a cofinal tower of finite regular cover. For the definitions of various $L^2$-invariants, see \cite{Lu2}.

\begin{thm}[\cite{Lu1}]\label{betti}
Let $X$ be a finite, connected CW-complex with residually finite fundamental group $\Gamma$. Let $\Gamma \supset \Gamma_1 \supset \cdots \supset \Gamma_n \supset \cdots$ be a nested sequence of finite index normal subgroups of $\Gamma$ with $\cap \Gamma_n=\{1\}$, and let $X_n$ be the finite cover of $X$ associated with $\Gamma_n\subset \Gamma$, then $$\lim_{n\rightarrow \infty}\frac{b_p(X_n)}{[\Gamma:\Gamma_n]}=b_p^{(2)}(X).$$
\end{thm}

Since finite volume hyperbolic $3$-manifolds have vanishing $L^2$-betti numbers (\cite{LL}), by applying the above result to hyperbolic $3$-manifolds, we have the following immediate corollary.

\begin{col}
For any hyperbolic $3$-manifold $M$ with finite volume, and any tower of finite regular covers $\cdots \rightarrow M_n \rightarrow \cdots \rightarrow M_1\rightarrow M$ with $\cap\pi_1(M_i)=\{1\}$, $$\lim_{n\rightarrow \infty}\frac{b_1(M_n)}{[\pi_1(M):\pi_1(M_n)]}=0.$$
\end{col}

Along with Agol's virtually infinite first betti number theorem (\cite{Ag}), these results imply that the first betti numbers of finite covers of a fixed hyperbolic $3$-manifold can go to infinity, but this trend does not grow very fast, which is a very interesting phenomenon.

On the other hand, a natural question is, whether the above approximation of the $L^2$-betti number can be generalized to some approximation of the $L^2$-torsion.

In particular, in \cite{LS}, L\"{u}ck and Schick showed that, for a finite volume hyperbolic $3$-manifold $M$, its $L^2$-torsion is related with its hyperbolic volume by the following equality: $$\rho^{(2)}(\widetilde{M})=-\frac{Vol(M)}{6\pi}.$$ So there arises the following natural question (see \cite{Lu2} Question 13.73, \cite{Lu3} Question 1.12 and \cite{BV}).

\begin{que}\label{torsiongrowth}
Let $M$ be a hyperbolic $3$-manifold with finite volume, does there exist a cofinal tower of finite regular covers $\cdots\rightarrow M_n\rightarrow \cdots \rightarrow M_1\rightarrow M$, such that $$\lim_{n\rightarrow \infty}\frac{\ln{|Tor(H_1(M_n;\mathbb{Z}))|}}{[\pi_1(M):\pi_1(M_n)]}=\frac{Vol(M)}{6\pi}?$$
\end{que}

In \cite{Le}, Le claim that $$\lim_{n\rightarrow \infty}\frac{\ln{|Tor(H_1(M_n;\mathbb{Z}))|}}{[\pi_1(M):\pi_1(M_n)]}\leq\frac{Vol(M)}{6\pi}$$ holds for any cofinal tower of finite regular covers.

If the answer of Question \ref{torsiongrowth} is yes, any hyperbolic $3$-manifold admits a certain cofinal tower of finite regular covers, with exponential growth on their homological torsion. However, a much weaker question, whether any hyperbolic $3$-manifold virtually has nontrivial homological torsion, was still unknown. In the survey paper \cite{AFW}, Aschenbrenner, Friedl and Wilton asked the following question.

\begin{que}\label{torsion}
Let $M$ be a hyperbolic $3$-manifold with finite volume, does $M$ admit a finite cover $N$ such that $Tor(H_1(N;\mathbb{Z}))\ne 0$?
\end{que}

This paper is devoted to answer Question \ref{torsion} for closed hyperbolic $3$-manifolds. Actually, we will prove that any closed hyperbolic $3$-manifold virtually contains any prescribed subgroup in its homological torsion.
\begin{thm}\label{main}
For any finite abelian group $A$, and any closed hyperbolic $3$-manifold $M$, $M$ admits a finite cover $N$, such that $A$ is a direct summand of $Tor(H_1(N;\mathbb{Z}))$.
\end{thm}

Since Agol showed that hyperbolic $3$-manifolds have virtually infinite first betti number (\cite{Ag}) and the first betti number does not decrease under taking finite cover, we have the following immediate corollary.
\begin{col}
For any finitely generated abelian group $A$, and any closed hyperbolic $3$-manifold $M$, $M$ admits a finite cover $N$, such that $A$ is a direct summand of $H_1(N;\mathbb{Z})$.
\end{col}

\begin{rem}
In a previous draft of this paper, the author used the result that $\pi_1(M)$ is LERF (\cite{Ag},\cite{Wi}), and only showed that $A$ embeds into $H_1(N;\mathbb{Z})$. Then Agol and Friedl informed the author about the virtual retract property of quasi-convex subgroups in $\pi_1(M)$ (\cite{HW}), then we could promote the result to make $A$ to be a direct summand.
\end{rem}

In the proof of Theorem \ref{main}, we will use Kahn and Markovic's construction of immersed almost totally geodesic surfaces in closed hyperbolic $3$-manifolds (\cite{KM1}). Since Kahn and Markovic's construction requires the manifold has a positive injectivity radius, it does not work for hyperbolic $3$-manifolds with cusps. So we can not show the same result for cusped $3$-manifolds, and we have the following natural question.
\begin{que}
Whether Theorem \ref{main} holds for finite volume hyperbolic $3$-manifolds with cusps?
\end{que}

\subsection{Sketch of the Proof}

In this paper, we will always use the symbols $\bold{l}$, $\bold{d}$ to denote the complex length and complex distance. The definitions of $\bold{l}$, $\bold{d}$, $\bold{hl}_{\Pi}$ and $s$ are given in Section \ref{pastwork}, which are standard notations in \cite{KM1}.

We will use Kahn and Markovic's construction of immersed almost totally geodesic surfaces in closed hyperbolic $3$-manifolds (\cite{KM1}) to do the following construction. For any closed hyperbolic $3$-manifold $M$ and any positive integer $p\geq 2$, we will construct an immersed $\pi_1$-injective $2$-complex $f:X_p\looparrowright M$, which provides us the virtual homological torsion.

More precisely, suppose $f:S\looparrowright M$ is a Kahn-Markovic surface, and $S$ is equipped with a pants decomposition $\mathcal{C}$. Then Kahn and Markovic's theorem implies that, there exist some small number $\epsilon>0$ and some large number $R>0$, such that for any simple closed curve $C\in \mathcal{C}$, $|\bold{hl}(C)-\frac{R}{2}|<\epsilon$ and $|s(C)-1|<\frac{\epsilon}{R}$ holds. The exponential mixing property of the frame flow (\cite{Mo},\cite{Po}) implies that there exists a closed geodesic $\gamma$ in $M$, such that $|\bold{l}(\gamma)-\frac{R+2\pi i}{p}|<\frac{\epsilon}{p}$. Moreover, we can choose $S$ such that $f(C)$ goes along$\gamma$ for $p$ times for some $C\in \mathcal{C}$.

By passing to a two-fold cover of $S$ if necessary, we can cut $S$ along $C$ to get a connected surface $S'$ with two oriented boundary components $C_1$ and $C_2$, with $[C_1]-[C_2]=0$ in $H_1(S';\mathbb{Z})$. Then $X_p$ is defined to be the quotient space of $S'$ under the $\frac{2\pi}{p}$-rotations on $C_1$ and $C_2$ respectively. Let $c_1$ and $c_2$ denote the image of $C_1$ and $C_2$ in $X_p$ respectively (with induced orientations), and we still use $f$ to denote the map $f:X_p\looparrowright M$ induced by the immersion $S\looparrowright M$.

Geometrically, in the closed hyperbolic $3$-manifold $M$, away from points in $c_1\cup c_2$, $f(X_p)$ locally looks like an almost totally geodesic surface. On a neighborhood of $f(c_i)$, $f(X_p)$ is almost a $(p\text{-prong})\times I$ with the top and bottom identified by the $\frac{2\pi}{p}$-rotation. Here the $p$-prong satisfies that any two adjacent edges have angle $\frac{2\pi}{p}$.

By doing cut-and-paste surgeries on $X_p$ with other Kahn-Markovic surfaces, we can assume that any essential arc in $X_p$ with end points in $c_1\cup c_2$ is very long. In this case we will show that $f:X_p\looparrowright M$ is $\pi_1$-injective.

Now we give two strategies to construct virtual homological torsions for closed hyperbolic $3$-manifolds. One strategy uses LERF and the other one uses the virtual retract property. The strategy using LERF can only give an embedding of the finite abelian group $A$ into $Tor(H_1(N,\mathbb{Z}))$ for some finite cover $N$; while the second strategy can show that $A$ is actually a virtual direct summand, which is stronger. However, the first strategy gives us an interesting codimension-$0$ submanifold in some finite cover $N$, which might be useful in some further research, so we give both strategies here.

{\bf Strategy I:} Let $\widetilde{M}$ be the infinite cover of $M$ associate to $f_*(\pi_1(X_p))\subset \pi_1(M)$, then $\widetilde{M}$ is a geometric finite hyperbolic $3$-manifold. Let $\hat{M}$ be a compact core of $\widetilde{M}$, then the boundary of $\hat{M}$ is incompressible, and there exists an order-$p$ element $\alpha=[c_1]-[c_2]\in H_1(\hat{M};\mathbb{Z})$. It is also easy to show that, in the order-$p$ subgroup of $H_1(\hat{M};\mathbb{Z})$ generated by $\alpha$, only $0$ and $\frac{p}{2}\alpha$ can be carried by $H_1(\partial\hat{M};\mathbb{Z})$ when $p$ is even, and only $0$ is carried by $H_1(\partial\hat{M};\mathbb{Z})$ when $p$ is odd.

Since fundamental groups of hyperbolic $3$-manifolds are LERF (\cite{Ag}, \cite{Wi}), by Scott's criterion of LERF (\cite{Sc}), there exists an intermediate finite cover $N\rightarrow M$ of $\widetilde{M}\rightarrow M$ such that $\hat{M}$ embeds into $N$. By an M-V sequence argument, $\hat{M}$ gives a $\mathbb{Z}_{\sigma(p)}$ subgroup in $Tor(H_1(N;\mathbb{Z}))$. Here $\sigma(p)=p$ when $p$ is an odd number, and $\sigma(p)=p/2$ if $p$ is even.

For two such geometrically finite subgroups $G_1=(f_1)_*(\pi_1(X_{p_1}))$ and $G_2=(f_2)_*(\pi_1(X_{p_2}))$, we can find $g\in \pi_1(M)$ such that both of the limit points of $g$ in $S^2_{\infty}$ do not lie in the limit sets $\Lambda(G_1)$ and $\Lambda(G_2)$. Then for a large enough positive integer $n$, the same argument as above shows that the geometric finite subgroup $G_1*g^nG_2g^{-n}\subset \pi_1(M)$ gives a $\mathbb{Z}_{\sigma(p_1)}\oplus \mathbb{Z}_{\sigma(p_2)}$ subgroup in the homology of some finite cover $N$. The result for a general finite abelian group $A$ can be shown by induction as the $\mathbb{Z}_{\sigma(p_1)}\oplus \mathbb{Z}_{\sigma(p_2)}$ case.

{\bf Strategy II:} Agol and Wise showed that $\pi_1(M)$ is virtually special (\cite{Ag},\cite{Wi}), so we can suppose that $\pi_1(M)$ is already the group of a special cube complex. Since quasi-convex subgroups of special groups are virtual retract (\cite{HW}), there exists a finite cover $N$ of $M$, such that the following conditions hold.

1) $\pi_1(X_p)\subset \pi_1(N)$.

2) For the inclusion map $i:\pi_1(X_p)\rightarrow \pi_1(N)$, there exists a retract homomorphism $r:\pi_1(N)\rightarrow \pi_1(X_p)$ such that $r\circ i=id_{\pi_1(X_p)}$.

The maps on fundamental groups induce maps on the first homology: $$H_1(X_p;\mathbb{Z})\xrightarrow{i_*} H_1(N;\mathbb{Z}) \xrightarrow{r_*} H_1(X_p;\mathbb{Z}).$$
Since $r_*\circ i_*=id$, we know that $H_1(X_p;\mathbb{Z})$ is a direct summand of $H_1(N;\mathbb{Z})$.

It is easy to compute that $H_1(X_p;\mathbb{Z})\cong \mathbb{Z}^{2g+1}\oplus \mathbb{Z}_p$, so $\mathbb{Z}_p$ is a direct summand of $H_1(N;\mathbb{Z})$.

For a general finite abelian group $A$, we can do the induction as in Strategy I and use the virtual retract property to construct our desired finite cover $N$.

This paper is organized as the following. In Section 2, we will give a quick review of Kahn and Markovic's result on constructing immersed almost totally geodesic surfaces in closed hyperbolic $3$-manifolds (\cite{KM1}), and prove some related lemmas. In Section 3, we will carry out the above discussion more concretely and rigorously, modulo the $\pi_1$-injectivity result (Theorem \ref{injectivity}). The $\pi_1$-injectivity property of $f:X_p\looparrowright M$ is a technical result and the proof will be deferred to Section 4.

{\bf Acknowledgement:} The author is grateful to his advisor David Gabai for many helpful conversations and suggestions. The author thanks Yi Liu for introducing this question to the author, and a few valuable conversations. The author would like to thank Ian Agol, Stefan Friedl and Vlad Markovic for comments on a previous draft of this paper. The author also thanks the referee for helpful comments and instructions.

\section{A Review of Kahn and Markovic's Works and Further Results}\label{pastwork}
In this section, we give a quick review of Kahn and Markovic's works on constructing immersed almost totally geodesic surfaces in closed hyperbolic $3$-manifolds (see \cite{KM1}). After introducing their works, we will develop a few related lemmas.

In \cite{KM1}, Kahn and Markovic proved the following Surface Subgroup Theorem, which is the first step to prove Thurston's Virtual Haken and Virtual Fibered Conjectures. (The conjectures were raised in \cite{Th2}, and settled in \cite{Ag}).

\begin{thm}[\cite{KM1}]\label{KMsurface}
For any closed hyperbolic $3$-manifold $M$, there exists an immersed closed hyperbolic surface $f:S\looparrowright M$, such that $f_*:\pi_1(S)\rightarrow \pi_1(M)$ is an injective map.
\end{thm}

Actually, the surfaces constructed in Theorem \ref{KMsurface} are almost totally geodesic surfaces, which are constructed by pasting oriented {\it good pants} together along oriented {\it good curves} in an almost totally geodesic way. In the following, we will describe Kahn and Markovic's construction with more details.

At first, we need to give some geometric definitions.

Let $\alpha$ be an oriented geodesic arc in a closed hyperbolic $3$-manifold with initial point $p$ and terminal point $q$. For two unit normal vectors $\vec{v}$ and $\vec{w}$ of $\alpha$ at $p$ and $q$ respectively, we define $\bold{d}_{\alpha}(\vec{v},\vec{w})$ by the following way. Let $\vec{v}'$ be the parallel transportation of $\vec{v}$ to $q$ along $\alpha$, $\theta \in \mathbb{R}/2\pi \mathbb{Z}$ be the oriented angle between $\vec{v}'$ and $\vec{w}$ (with respect to the orientation of $\alpha$), and the length of $\alpha$ be $l>0$. Then the complex distance between $\vec{v}$ and $\vec{w}$ along $\alpha$ is defined to be $\bold{d}_{\alpha}(\vec{v},\vec{w})=l+\theta i$.

For an oriented closed geodesic $\gamma$ in a hyperbolic $3$-manifold, we define its complex length in a similar way. Choose an arbitrary point $p$ on $\gamma$ and a unit normal vector $\vec{v}$ of $\gamma$ at $p$, then we can consider $\gamma$ as an oriented geodesic arc from $p$ to $p$. Then the complex length of $\gamma$ is defined to be $\bold{l}(\gamma)=\bold{d}_{\gamma}(\vec{v},\vec{v})$. This complex length not only measures the length of $\gamma$ in the usual sense, but also measures the rotation angle of the corresponding hyperbolic isometry. Note that the complex length of a closed geodesic does not depend on the orientation and the choices we made.

In the following, we will use $\Pi^0$ to denote the oriented pair of pants.

\begin{defn}
For a closed hyperbolic $3$-manifold $M$, a map $f:\Pi^0\rightarrow M$ is called a {\it skew pair of pants} if $f_*:\pi_1(\Pi^0) \rightarrow \pi_1(M)$ is injective, and $f(\partial \Pi^0)$ is a union of three closed geodesics.
\end{defn}

We will always think about homotopic skew pair of pants as the same object, and we will use $\Pi$ to denote a skew pair of pants $f:\Pi^0\rightarrow M$ when it does not cause any confusion.

Let $C_1$, $C_2$ and $C_3$ be the three oriented boundary components of $\Pi^0$, then let $\gamma_1$, $\gamma_2$ and $\gamma_3$ be the three oriented closed geodesics $f(C_1)$, $f(C_2)$ and $f(C_3)$ respectively. Let $a_i$ be the simple arc on $\Pi^0$ which connects $C_{i-1}$ and $C_{i+1}$, such that $a_1$, $a_2$, and $a_3$ are disjoint with each other (they are called seams of $\Pi_0$). Then we can assume that $f(a_i)$ is a geodesic arc perpendicular with both $\gamma_{i-1}$ and $\gamma_{i+1}$ for $i=1,2,3$, and denote $f(a_i)$ by $\eta_i$.

Now we fix one $\gamma_i$, and give orientations for $\eta_{i-1}$ and $\eta_{i+1}$ such that they are both pointing away from $\gamma_i$. Then $\eta_{i-1}$ and $\eta_{i+1}$ divide $\gamma_i$ to two oriented geodesic arcs $\gamma_i^{1}$ and $\gamma_i^{2}$, such that the orientation on $\gamma_i^{1}$ goes from $\eta_{i-1}\cap \gamma_i$ to $\eta_{i+1}\cap \gamma_i$. Let $\vec{v}_{i-1}$ and $\vec{v}_{i+1}$ be the unit tangent vectors of $\eta_{i-1}$ and $\eta_{i+1}$ at $\eta_{i-1}\cap \gamma_i$ and $\eta_{i+1}\cap \gamma_i$ respectively, then we have a pair of vectors $(\vec{v}_{i-1},\vec{v}_{i+1})$ on the unit normal bundle $N^1(\gamma_i)$, which is called the pair of feet of $\Pi$ on $\gamma_i$. The hyperbolic geometry of right-angled hexagons in $\mathbb{H}^3$ implies that $\bold{d}_{\gamma_i^1}(\vec{v}_{i-1},\vec{v}_{i+1})=\bold{d}_{\gamma_i^2}(\vec{v}_{i+1},\vec{v}_{i-1})$.
So we can define the half length of $\gamma_i$ with respect to $\Pi$ by $$\bold{hl}_{\Pi}(C_i)=\bold{d}_{\gamma_i^1}(\vec{v}_{i-1},\vec{v}_{i+1})=\bold{d}_{\gamma_i^2}(\vec{v}_{i+1},\vec{v}_{i-1}).$$

Now we are ready to define {\it good curves} and {\it good pants}.

\begin{defn}
Fix a small number $\epsilon>0$ and a large number $R>0$. For a closed oriented geodesic $\gamma$ in $M$, we say $\gamma$ is an {\it $(R,\epsilon)$-good curve} if $|\bold{l}(\gamma)-R|<2\epsilon$. The set of $(R,\epsilon)$-good curves is denoted by $\bold{\Gamma}_{R,\epsilon}$.

For a skew pair of pants $f:\Pi^0\rightarrow M$, we say it is an {\it $(R,\epsilon)$-good pants} if $|\bold{hl}_{\Pi}(C)-\frac{R}{2}|<\epsilon$ holds for all the three cuffs (boundary components) of $\Pi^0$. The set of $(R,\epsilon)$-good pants is denoted by $\bold{\Pi}_{R,\epsilon}$.
\end{defn}

In the following, we will work with a very small number $\epsilon>0$ and a very large number $R>0$, and the precise value of $\epsilon$ and $R$ will be determined later. When $R$ and $\epsilon$ have been fixed, we will only talk about good curves and good pants, instead of $(R,\epsilon)$-good curves and $(R,\epsilon)$-good pants, when it does not cause any confusion. Note that oriented boundary components of $(R,\epsilon)$-good pants are $(R,\epsilon)$-good curves.

For a good curve $\gamma\in \bold{\Gamma}_{R,\epsilon}$, the normal bundle $N^1(\gamma)$ of $\gamma$ is naturally identified with $\mathbb{C}/\bold{l}(\gamma)\mathbb{Z}+2\pi i \mathbb{Z}$. If we have a skew pair of pants $\Pi$ which has $\gamma$ as one of its oriented boundary component, we can define the half normal bundle of $\gamma$ by $N^1(\sqrt{\gamma})=\mathbb{C}/\bold{hl}_{\Pi}(\gamma)\mathbb{Z}+2\pi i \mathbb{Z}$. Then the pair of feet of $\Pi$ on $\gamma$ are identified to one point in $N^1(\sqrt{\gamma})$, which is called the {\it foot} of $\Pi$ on $\gamma$, and denoted by $foot_{\gamma}(\Pi)$.

Now we are ready to talk about maps from surfaces to closed hyperbolic $3$-manifolds. In the following, we will fix a closed hyperbolic $3$-manifold and work on it.
 
Suppose $S$ is a compact oriented closed surface with negative Euler characteristic, equipped with a pants decomposition $\mathcal{C}$. Then the closure of each component of $S\setminus \mathcal{C}$ is an oriented pair of pants, and we call such a component a pants in $S$.

\begin{defn}
A map $f:S\rightarrow M$ is called {\it viable} if the following conditions hold.
\begin{itemize}
\item For each pants $\Pi$ in $S$, $f|_{\Pi}:\Pi \rightarrow M$ is a skew pair of pants.
\item For any two pants $\Pi$ and $\Pi'$ in $S$ sharing a curve $C\in \mathcal{C}$, $\bold{hl}_{\Pi}(C)=\bold{hl}_{\Pi'}(C)$ holds.
\end{itemize}
\end{defn}

So for a viable map $f:S\rightarrow M$, we will use $\bold{hl}(C)$ to denote $\bold{hl}_{\Pi}(C)$ for each $C\in \mathcal{C}$. For two pants in $S$ sharing a curve $C\in \mathcal{C}$, we give $C$ an arbitrary orientation. Let $\Pi$ be the pants lies to the left of $C$ on $S$, and $\Pi'$ lies to the right. Let $\gamma=f(C)$, and $\bar{\gamma}$ be the same closed geodesic with the opposite orientation, then we can compare the feet of $\Pi$ and $\Pi'$ on $N^1(\sqrt{\gamma})$ by the following shearing parameter (here $N^1(\sqrt{\gamma})$ and $N^1(\sqrt{\bar{\gamma}})$ are naturally identified with each other):
$$s(C)=foot_{\gamma}(f|_{\Pi})-foot_{\bar{\gamma}}(f|_{\Pi'})-\pi i \in N^1(\sqrt{\gamma})=\mathbb{C}/\bold{hl}(C)\mathbb{Z}+2\pi i \mathbb{Z}.$$

Now we can precisely describe the immersed almost totally geodesic surfaces constructed in \cite{KM1}.

\begin{thm}[\cite{KM1}]\label{KM1}
For any closed hyperbolic $3$-manifold $M$, there exists constants $q>0$ and $K>0$, such that for every small enough $\epsilon>0$ and every large enough $R>0$, the following statement holds. There exists a closed surface $S$ equipped a pants decomposition $\mathcal{C}$, and a viable map $f:S\rightarrow M$ such that for any $C\in \mathcal{C}$, we have
\begin{equation}\label{KMcondition}
\left\{ \begin{array}{l}
        |\bold{hl}(C)-\frac{R}{2}|<\epsilon, \\
        |s(C)-1|<KRe^{-qR}<\frac{\epsilon}{R}.
\end{array} \right.
\end{equation}
Moreover, $f_*:\pi_1(S)\rightarrow \pi_1(M)$ is injective.
\end{thm}

We will call a viable map $f:S\rightarrow M$ an {\it $(R,\epsilon)$-almost totally geodesic surface}, if the inequality \eqref{KMcondition} holds for each $C\in \mathcal{C}$.

The existence of such $(R,\epsilon)$-almost totally geodesic closed surfaces is proved by the following strategy in \cite{KM1}. For any good curve $\gamma$, one can consider all the good pants in $M$ with $\gamma$ as one of its oriented boundary component, then consider all the feet $foot_{\gamma}(\Pi)$ on $N^1(\sqrt{\gamma})$. Kahn and Markovic showed that these feet on $N^1(\sqrt{\gamma})$ are very equidistributed, so they can paste all the good pants together in a proper way such that $|s(C)-1|<\frac{\epsilon}{R}$ holds.

More precisely, Kahn and Markovic constructed an integer valued measure $\mu_0$ on $\bold{\Pi}_{R,\epsilon}$, with the following nice property.

\begin{prop}[\cite{KM1}]\label{measure}
There exists an integer valued $\mu_0$ on $\bold{\Pi}_{R,\epsilon}$ with the following properties. Let $\hat{\partial}\mu_0$ be the counting measure on $$N^1(\sqrt{\bold{\Gamma}_{R,\epsilon}})=\bigcup_{\gamma\in \bold{\Gamma}_{R,\epsilon}}N^1(\sqrt{\gamma})$$ given by the feet of pants in $\bold{\Pi}_{R,\epsilon}$ and weighted by $\mu_0$. Then for any $\gamma \in\bold{\Gamma}_{R,\epsilon}$, there exists a constant $K_{\gamma}\geq 0$, such that $\hat{\partial}\mu_0|_{N^1(\sqrt{\gamma})}$ is $KRe^{-qR}$-equivalent to $K_\gamma \lambda$ for some universal constant $K>0$. Here $\lambda$ is the standard Lebesgue measure on $N^1(\sqrt{\gamma})\cong \mathbb{C}/\bold{hl}(\gamma)\mathbb{Z}+2\pi i\mathbb{Z}$.
\end{prop}

For two Borel measures $\mu$ and $\nu$ on a compact metric space $X$, we say $\mu$ and $\nu$ are {\it $\delta$-equivalent} for some $\delta>0$ if the following conditions hold.
\begin{itemize}
\item $\mu(X)=\nu(X)$.
\item For any Borel measurable subset $A\subset X$, $\mu(A)<\nu(N_{\delta}(A))$ holds. Here $N_{\delta}(A)$ is the $\delta$-neighborhood of $A$ in $X$.
\end{itemize}

In the proof of the existence of good pants (curves) and the equidistribution result, the following exponential mixing property of the frame flow played a crucial role.

\begin{thm}[\cite{Mo},\cite{Po}]\label{mixing}
Let $M$ be a closed hyperbolic $3$-manifold, $\mathcal{F}(M)$ be the frame bundle of $M$, $\Lambda$ be the Liouville measure on $\mathcal{F}(M)$ which is invariant under the frame flow $g_t:\mathcal{F}(M)\rightarrow \mathcal{F}(M)$.

Then there exists a constant $q>0$ that depends only on $M$, such that the following statement holds. Let $\psi,\phi: \mathcal{F}(M)\rightarrow \mathbb{R}$ be two $C^1$ functions, then for any $r\in \mathbb{R}$, $$\Bigl\lvert \Lambda\left(\mathcal{F}(M)\right)\int_{\mathcal{F}(M)}(g_r^*\psi)(x)\phi(x)d\Lambda(x)-\int_{\mathcal{F}(M)}\psi(x)d\Lambda(x)\int_{\mathcal{F}(M)}\phi(x)d\Lambda(x)\Bigr\rvert \leq Ce^{-q|r|}.$$ Here $C>0$ only depends on the $C^1$-norms of $\psi$ and $\phi$.
\end{thm}

For technical reasons, we need a slightly stronger condition than \eqref{KMcondition} in this paper, so we need the following proposition. In \cite{Sa}, \v{S}ari\'{c} has shown that we can strengthen $|\bold{hl}(C)-\frac{R}{2}|<\epsilon$ to $|\bold{hl}(C)-\frac{R}{2}|<\frac{\epsilon}{R}$, and it is the essential part of the proof, so we only give a very brief proof here.

\begin{prop}\label{Saric}
For any closed hyperbolic $3$-manifold $M$, there exists a constant $q>0$ and a polynomial $P(\cdot)$, such that for every small enough $\epsilon>0$ and large enough $R>0$, the following statement holds. There exists a closed surface $S$ equipped with a pants decomposition $\mathcal{C}$, and a viable map $f:S\rightarrow M$, such that for any $C\in \mathcal{C}$, we have
\begin{equation}\label{prom}
\left\{ \begin{array}{l}
        |\bold{hl}(C)-\frac{R}{2}|<\frac{\epsilon}{R}, \\
        |s(C)-1|<P(R)e^{-qR}<\frac{\epsilon}{R^2}.
\end{array} \right.
\end{equation}
Moreover, $f_*:\pi_1(S)\rightarrow \pi_1(M)$ is injective.
\end{prop}

\begin{proof}
In the introduction of \cite{Sa}, \v{S}ari\'{c} pointed out that $|\bold{hl}(C)-\frac{R}{2}|<\epsilon$ can be replaced by $|\bold{hl}(C)-\frac{R}{2}|<\frac{\epsilon}{R}$. The main reason that such a refinement is applicable is, the exponential mixing property of frame flow (\cite{Mo}, \cite{Po}) gives the exponential rate, which beats any polynomial rate.

More precisely, in Kahn and Markovic's construction in \cite{KM1}, the following function is a crucial ingredient. For an arbitrary point $F_0$ in the frame bundle $\mathcal{F}(\mathbb{H}^3)$, we can choose a $C^1$ bump function $f_{\epsilon}^{F_0}:\mathcal{F}(\mathbb{H}^3)\rightarrow \mathbb{R}_{\geq 0}$ supporting on the $\epsilon$-neighborhood of $F_0$ in $\mathcal{F}(\mathbb{H}^3)$ (here $\epsilon>0$ is smaller than the injectivity radius of $M$), such that $$\int_{\mathcal{F}(\mathbb{H}^3)}f_{\epsilon}^{F_0}(x)d\Lambda(x)=1.$$ By pulling back $f_{\epsilon}^{F_0}$ by $Isom_{+}(\mathbb{H}^3)$ and projecting to $\mathcal{F}(M)$, we get a function $f_{\epsilon}^{F}:\mathcal{F}(M)\rightarrow \mathbb{R}_{\geq 0}$ centered at $F$ for each $F\in \mathcal{F}(M)$. Then Kahn and Markovic's constructions of $(R,\epsilon)$-good pants and immersed almost totally geodesic surfaces start from the function $f_{\epsilon}$.

In \cite{Sa}, \v{S}ari\'{c} gave the following observation. For the time $t$ frame flow, we consider an alternative bump function $f_{\frac{\epsilon}{t}}$. By taking $f_{\frac{\epsilon}{t}}(x)$ to be $t^6\cdot f_{\epsilon}(xt)$ up to a constant close to $1$, we can suppose $\int_{\mathcal{F}(\mathbb{H}^3)}f_{\frac{\epsilon}{t}}(x)d\Lambda(x)=1$. Since the frame flow has exponential mixing rate in term of $t$, while the constant $C$ in Theorem \ref{mixing} can be estimated by the $H_2^2$-Sobolev norm of $f_{\frac{\epsilon}{t}}$, which grows in a polynomially rate, so we have \begin{equation}\label{prompt}
\Bigl\lvert \Lambda\left(\mathcal{F}(M)\right)\int_{\mathcal{F}(M)}(g_t^*f_{\frac{\epsilon}{t}}^{F_1})(x)f_{\frac{\epsilon}{t}}^{F_2}(x)d\Lambda(x)-1\Bigr\rvert \leq P(t)e^{-q|t|}\rightarrow 0
\end{equation}
when $t$ goes to $\infty$. Then all the works in \cite{KM1} are still available under the inequality \eqref{prompt}, and $|\bold{hl}(C)-\frac{R}{2}|<\frac{\epsilon}{R}$ holds.

For the second inequality in \eqref{KMcondition}, under the new bump function $f_{\frac{\epsilon}{R}}$ and inequality \eqref{prompt}, the same argument as in \cite{KM1} gives an integer valued measure $\mu_0$ on $\bold{\Pi}_{R,\frac{\epsilon}{R}}$, such that $\hat{\partial} \mu_0|_{N^1(\sqrt{\gamma})}$ is $P'(R)e^{-qR}$-equivalent to $K_{\gamma}\lambda$ for each $\gamma \in \bold{\Gamma}_{R,\frac{\epsilon}{R}}$, with $P'$ being a polynomial and $K_{\gamma}\geq 0$.

So we can get a gluing of good pants with $|s(C)-1|<P(R)e^{-qR}<\frac{\epsilon}{R^2}$, and the proof of this proposition is done.
\end{proof}

We also need the following two lemmas which are basic for our construction, and their proofs are closely related with Kahn and Markovic's works in \cite{KM1}.

\begin{lem}\label{geodesic}
For any positive integers $p\geq 2$ and $q\geq 1$, then for small enough $\epsilon>0$ and large enough $R>0$, there exists a closed geodesic $\gamma'$ in $M$, such that $|\bold{l}(\gamma')-\frac{R+2\pi i}{p}|<\frac{\epsilon}{qR}$.
\end{lem}
\begin{proof}
Take an arbitrary frame $F=(p,\vec{v},\vec{n})$ in $\mathcal{F}(M)$, with $p\in M$, $\vec{v},\vec{n}\in T^1_p(M)$ with $\vec{v} \perp \vec{n}$. Let $\vec{n}'$ be the $\frac{2\pi}{p}$ rotation of $\vec{n}$ along $\vec{v}$, and let $F'=(p,\vec{v},\vec{n}')$.

Now we consider functions $f_{\frac{\epsilon}{4qR}}^F$ and $f_{\frac{\epsilon}{4qR}}^{F'}$ as in the proof of Proposition \ref{Saric}. By applying equation \eqref{prompt}, with $t$ replaced by $\frac{R}{p}$ and $\epsilon$ replaced by $\frac{\epsilon}{4pq}$, there exists an oriented geodesic arc $\alpha$ in $M$, with two frames $\hat{F}$ and $\hat{F}'$ at its initial and terminal points respectively, such that the following conditions hold.
\begin{enumerate}
\item The parallel transportation of $\hat{F}$ along $\alpha$ to its terminal point equals $\hat{F}'$.
\item The first vector component of $\hat{F}$ is tangent to $\alpha$.
\item The distance between $F$ and $\hat{F}$, and the distance between $F'$ and $\hat{F}'$ in $\mathcal{F}(M)$ are both smaller than $\frac{\epsilon}{4qR}$.
\end{enumerate}

Then by connecting the initial and terminal points of $\alpha$ by an $\frac{\epsilon}{2qR}$-short geodesic in $M$, we get a closed path which is homotopic to a closed geodesic $\gamma'$. Then this $\gamma'$ satisfies $|\bold{l}(\gamma')-\frac{R+2\pi i}{p}|<\frac{\epsilon}{qR}$, by elementary estimations in the hyperbolic geometry.
\end{proof}

\begin{lem}\label{pass}
There exists a universal constant $D>0$, such that for any small enough $\epsilon>0$ and large enough $R>0$, the following statement holds. For any closed geodesic $\gamma\in \bold{\Gamma}_{R,\frac{\epsilon}{DR}}$, there exists a closed surface $S$ with a pants decomposition $\mathcal{C}$, such that there exists an $(R,\frac{\epsilon}{R})$-almost totally geodesic immersion $f:S\looparrowright M$, and $f(C)=\gamma$ for some $C\in \mathcal{C}$.
\end{lem}

\begin{proof}
In \cite{KM1}, the integer valued measure $\mu_0$ in Proposition \ref{measure} is given by a real valued measure $\mu$ on $\bold{\Pi}_{R,\epsilon}$ ($\bold{\Pi}_{R,\frac{\epsilon}{R}}$ in our case), by first perturbing $\mu$ to a rational valued measure, then take an integer multiple.

So we need only to show that $\hat{\partial}\mu(N^1(\sqrt{\gamma}))>0$. In this case, we can take a small enough perturbation of $\mu$ so that $\hat{\partial}\mu_0(N^1(\sqrt{\gamma}))>0$ holds. Then in the construction of almost totally geodesic surfaces instructed by $\mu_0$, we must use some good pants with one of its cuff being $\gamma$. Then we take the component of the Kahn-Markovic surface which contains this good pants.

By the proof in Section 4.8 of \cite{KM1}, $\hat{\partial}\mu(N^1(\sqrt{\gamma}))>0$ if and only if there exist two frames $F_1=(p_1,\vec{v}_1,\vec{n}_1)$ and $F_2=(p_2,\vec{v}_2,\vec{n}_2)$, and two geodesic arcs $\alpha_1$ and $\alpha_2$ in $M$, such that the following conditions hold.
\begin{enumerate}
\item $\alpha_1$ has initial point $p_1$ and terminal point $p_2$, while $\alpha_2$ has initial point $p_2$ and terminal point $p_1$.
\item $\alpha_1\alpha_2$ is homotopic to $\gamma$.
\item Let $\omega(F_1)=(p_1,\omega(\vec{v}_1),\vec{n}_1)$ be the $\frac{2\pi}{3}$-rotation of $F_1$ with respect to $\vec{n}_1$, and $\bar{\omega}(F_2)=(p_2,\bar{\omega}(\vec{v}_2),\vec{n}_2)$ be the $\frac{4\pi}{3}$-rotation of $F_2$ with respect to $\vec{n}_2$, then both $a_{\alpha_1}(F_1,F_2)$ and $a_{\alpha_2}(\omega(F_1),\bar{\omega}(F_2))$ are positive.
\end{enumerate}

Under the modified Kahn-Markovic condition \eqref{prom}, for two frames $F_1$ and $F_2$ in $M$ with a geodesic arc $\alpha$ connecting their base points, $a_{\alpha}(F_1,F_2)$ is defined by the following way. Take two frames $\hat{F}_1$ and $\hat{F}_2$ in $\mathcal{F}(\mathbb{H}^3)$ projecting to $F_1$ and $F_2$ respectively, such that the geodesic arc in $\mathbb{H}^3$ connecting the base points of $\hat{F}_1$ and $\hat{F}_2$ projects to $\alpha$. Let $r=\frac{R}{2}+\ln{\frac{4}{3}}$, and let $g_{\frac{r}{4}}(\hat{F}_1)=(p'_1,\vec{v}'_1,\vec{n}'_1)$, $g_{\frac{r}{4}}(\hat{F}_2)=(p'_2,\vec{v}'_2,\vec{n}'_2)$, then $a_{\alpha}(F_1,F_2)$ is defined by:
$$a_{\alpha}(F_1,F_2)=\int_{\mathcal{F}(M^3)}(g_{\frac{r}{2}}^*f_{\frac{\epsilon}{DR}}^{(p'_1,\vec{v}'_1,\vec{n}'_1)})(x)f_{\frac{\epsilon}{DR}}^{(p'_2,-\vec{v}'_2,\vec{n}'_2)}(x)d\Lambda(x).$$
Here $D>0$ is some universal constant.

Then it is easy to check that if $\gamma\in \bold{\Gamma}_{R,\frac{\epsilon}{DR}}$, frames $F_1$ and $F_2$ satisfying the above conditions do exist.
\end{proof}

\section{Geometric Constructions}

In this section, for each closed hyperbolic $3$-manifold $M$, we will construct an immersed $\pi_1$-injective $2$-complex $X_p\looparrowright M$ which has good pants as its building blocks. Here $X_p$ is a local model of the homological $\mathbb{Z}_p$-torsion. Then we will show that the immersion $X_p\looparrowright M$ provides homological torsion in some finite cover of $M$.

\subsection{Construction of a $2$-complex}

We first give a brief sketch of our construction $X_p\looparrowright M$.

At first, there exists a Kahn-Markovic surface $f:S\looparrowright M$, such that for some $C\in\mathcal{C}$, $f(C)$ goes along some closed geodesic $\gamma'$ for $p$ times with $\bold{l}(\gamma')$ close to $\frac{R+2\pi i}{p}$. Then we cut $S$ along $C$, and quotient the two boundary components by $\frac{2\pi}{p}$-rotations, to get an immersed $2$-complex $X_p\looparrowright M$. By doing cut-and-past surgeries, we can make sure that the singular curves on $X_p$ are far away from each other, which guarantees that the immersion $X_p\looparrowright M$ is $\pi_1$-injective.

Now we fix a closed hyperbolic $3$-manifold $M$, and work with some very small $\epsilon>0$ and very large $R>0$ which will be determined later. Here we divide the construction into a few steps.

{\bf Step I.} By Lemma \ref{geodesic}, for any positive integer $p \geq 2$, there exists a closed geodesic $\gamma'$ in $M$ with $|\bold{l}(\gamma')-\frac{R+2\pi i}{p}|<\frac{\epsilon}{pDR}$ (here $D>0$ is the constant in Lemma \ref{pass}). Let $\gamma$ be the closed geodesic which travels around $\gamma'$ for $p$ times, then $\gamma$ is a nonprimitive closed geodesic with $|\bold{l}(\gamma)-R|<\frac{\epsilon}{DR}$, so $\gamma\in \bold{\Gamma}_{R,\frac{\epsilon}{DR}}$.

{\bf Step II.} By Lemma \ref{pass}, there exists an immersed $(R,\frac{\epsilon}{R})$-almost totally geodesic closed surface $f:S\looparrowright M^3$, such that for the corresponding pants decomposition $\mathcal{C}$ of $S$, there exists $C\in \mathcal{C}$ such that $f(C)=\gamma$. By taking a two-fold cover of $S$ if necessary, we can suppose that $C$ is a non-separating curve on $S$ and the two pants adjacent to $C$ are distinct. Let $S'$ be the surface obtained from $S$ by cutting along $C$, then $S'$ has an induced pants decomposition $(S',\mathcal{C'})$ (here $\mathcal{C'}$ does not contain the boundary of $S'$). Let $C_1$ and $C_2$ be the two boundary components of $S'$, and they are given orientations such that $[C_1]-[C_2]=0\in H_1(S';\mathbb{Z})$.

{\bf Step III.} Let $\rho_i:C_i\rightarrow C_i$, $i=1,2$ be the $\frac{2\pi}{p}$-rotation on the circle. Then we define $X_p$ to be the $2$-complex obtained from $S'$ quotient by the $\rho_i$-action for $i=1,2$, and let $c_i$ be the oriented embedded circle in $X_p$ which is the image of $C_i$. Since $C_1$ and $C_2$ are both mapped to $\gamma'$ for $p$ times, $f:S\looparrowright M$ induces a map $f:X_p\looparrowright M$ with $c_i$ mapped to $\gamma'$ for $i=1,2$. Then the pants decomposition $\mathcal{C'}$ on $S'$ and two curves $C_1,C_2$ induce a "pants decomposition" on $X_p$, which is denoted by $(X_p,\mathcal{C'},\{C_1, C_2\})$. Note that $C_1$ and $C_2$ are not embedded curves in $X_p$.

{\bf Step IV.} Now we define a graph $G(X_p)$ from $(X_p,\mathcal{C'},\{C_1, C_2\})$. Vertices of $G(X_p)$ are pants in $X_p$, two vertices are connected by an edge if the corresponding two pants share some $C\in \mathcal{C'}$. $G(X_p)$ is a trivalent graph except at two vertices $v_1,v_2$. These two vertices correspond with the two pants in $X_p$ containing $c_1$ and $c_2$ respectively, and both of them are degree-$2$ vertices.

By a path in a graph $G$, we mean a sequence of oriented edges in $G$, such that for two adjacent edges $e$ and $e'$ in the path, the terminal vertex of $e$ equals the initial vertex of $e'$. For a path in $G$, its (combinatorial) length is defined to be the number of oriented edges it contains, counted with multiplicity. We say a path in $G$ is inessential, if its initial and terminal vertices are the same vertex $v$, and the corresponding map between topological spaces $(I,\partial I)\rightarrow (G,\{v\})$ is homotopic to the constant map. We say a path is essential if it is not inessential.

Let $l(G(X_p))$ be the length of the shortest essential path in $G(X_p)$ with end points in $\{v_1,v_2\}$, and $n(G(X_p))$ be the number of such paths. We define the complexity of $G(X_p)$ to be $c(G(X_p))=\left(l(G(X_p)),-n(G(X_p))\right)$, and we will do inductive constructions to make $G(X_p)$ more and more complicated until $l(G(X_p))>Re^{\frac{R}{4}}$.

For any shortest essential path $\alpha$ with length $l\geq 1$, let $k=[\frac{l+1}{2}]$ and $e$ be the $k$-th edge on $\alpha$. Let $C_0\in \mathcal{C'}$ be the curve in the pants decomposition of $X_p$ corresponding with $e$, and let $C_0'$ denote the corresponding curve in $S$. Take a copy of $S$, and pass to a two-fold cover if necessary, such that $C_0'$ is a non-separating curve in $S$. Then we cut $X_p$ and $S$ along $C_0$ and $C_0'$ respectively, and re-paste them together to get a connected $2$-complex $X_p'$ with an induced pants decomposition $\mathcal{C''}$ (such kind of surgeries have appeared in \cite{KM2}). Since the pants in $S$ and $X_p$ have the same feet on $N^1(\sqrt{f(C_0)})$, we still have
$$
\left\{ \begin{array}{l}
        |\bold{hl}(C)-\frac{R}{2}|<\frac{\epsilon}{R}, \\
        |s(C)-1|<\frac{\epsilon}{R^2},
\end{array} \right.
$$
for any $C\in\mathcal{C''}$.

After this surgery, the shortest essential pathes in $G(X_p)$ going through the edge $e$ have been broken, and the corresponding length increases at least by $1$ in $G(X_p')$. So the complexity $c(G(X_p'))=\left(l(G(X_p')),-n(G(X_p'))\right)$ is bigger than $c(G(X_p))=\left(l(G(X_p)),-n(G(X_p))\right)$, i.e. either $l(G)$ increases, or $l(G)$ does not change and $-n(G)$ increases. After finitely many steps of such constructions, we can assure $l(G)>Re^{\frac{R}{4}}$. For simplicity, we still denote the $2$-complex by $X_p$, and denote the pants decomposition by $(X_p,\mathcal{C'},\{C_1,C_2\})$.

\begin{defn}\label{viable}
A representation $\rho:\pi_1(X_p)\rightarrow PSL_2(\mathbb{C})$ is called a {\it viable representation} if the following conditions hold.
\begin{enumerate}
\item For each $C\in \mathcal{C}\cup\{C_1,C_2\}$, let $g_C$ be a generator of $\pi_1(C)$, then $\rho(g_C)$ is a hyperbolic element in $PSL_2(\mathbb{C})$.
\item For each pants $\Pi$ in $X_p$, $\rho|_{\pi_1(\Pi)}$ is an injective map, and $\rho(\pi_1(\Pi))$ is a discrete subgroup of $PSL_2(\mathbb{C})$.
\item For any two pants $\Pi, \Pi'$ sharing some $C\in \mathcal{C}\cup\{C_1,C_2\}$, $\bold{hl}_{\Pi}(C)=\bold{hl}_{\Pi'}(C)$ holds.
\end{enumerate}
Note that although $\rho(\pi_1(X_p))$ may not be a discrete subgroup of $PSL_2(\mathbb{C})$, $\bold{hl}_{\Pi}(C)$ can still be defined.

A map $f:X_p\rightarrow M$ is called a {\it viable map} if $f_*:\pi_1(X_p)\rightarrow \pi_1(M)\subset PSL_2(\mathbb{C})$ is a viable representation.
\end{defn}

As a summary of the above construction, we have the following proposition which guarantees the existence of an immersed $2$-complex $X_p\looparrowright M$.

\begin{prop}\label{mappedcomplex}
For any closed hyperbolic $3$-manifold $M$, and any positive integer $p\geq 2$, there exists a constant $\hat{\epsilon}>0$, such that for any $0<\epsilon<\hat{\epsilon}$ and any $R$ sufficiently large, the following statement holds. There exists a $2$-complex $X_p$ as above with a pants decomposition $(X_p,\mathcal{C'}, \{C_1, C_2\})$, and a viable map $f:X_p \looparrowright M$ such that the following conditions hold.
\begin{enumerate}
\item The induced graph $G(X_p)$ satisfies $l(G(X_p))>Re^{\frac{R}{4}}$.
\item For any $C\in \mathcal{C'}\cup \{C_1, C_2\}$, $|\bold{hl}(C)-\frac{R}{2}|<\frac{\epsilon}{R}$.
\item For any $C \in \mathcal{C'}$, $|s(C)-1|<\frac{\epsilon}{R^2}$.
\item Let $c_k$ be the image of $C_k$ in the $2$-complex $X_p$, then $|\bold{l}(c_k)-\frac{R+2\pi i}{p}|<\frac{\epsilon}{pR}$ for $k=1,2$.
\end{enumerate}
\end{prop}

\begin{rem}
The construction of $X_p\looparrowright M$ is very similar to the construction in \cite{KM2}.

In \cite{KM2}, Kahn and Markovic constructed immersed quasi-fuchsian surfaces $S\looparrowright M$ by pasting (generalized) good pants, the immersion satisfies inequality \eqref{KMcondition} for curves $C\in \mathcal{C}$, except bending along a sparse collection of curves in $\mathcal{C}$ that are far away from each other. Our construction is almost following the same idea with theirs, but we construct immersed $2$-complexes, instead of surfaces.

Moreover, the cut-and-paste technique in Step IV of our construction also appeared in \cite{KM2}. In \cite{KM2}, Kahn and Markovic amalgamated two immersed almost totally geodesic surfaces to one immersed quasi-fushsian surface, and they used the cut-and-paste technique to make sure that the bending curves are far away from each other.
\end{rem}

\subsection{Finite Cyclic Subgroups in Virtual Homology}

The following theorem is the most technical theorem in this paper, which is an analogue of Theorem 2.2 in \cite{KM1}.

\begin{thm}\label{injectivity}
There are universal constants $\hat{\epsilon}>0$ and $\hat{R}>0$ depend only on $p$ and $M$, such that for any $0<\epsilon<\hat{\epsilon}$ and any $R>\hat{R}>0$, the following statement holds. If $X_p$ is a $2$-complex with a pants decomposition $(X_p,\mathcal{C'},\{C_1, C_2\})$ constructed as last section, and $\rho:\pi_1(X_p) \rightarrow PSL_2(\mathbb{C})$ is a viable representation such that the following conditions hold.
\begin{enumerate}
\item The induced graph $G(X_p)$ satisfies $l(G(X_p))>Re^{\frac{R}{4}}$.
\item For any $C\in \mathcal{C'}\cup \{C_1, C_2\}$, $|\bold{hl}(C)-\frac{R}{2}|<\frac{\epsilon}{R}$.
\item For any $C \in \mathcal{C'}$, $|s(C)-1|<\frac{\epsilon}{R^2}$.
\item $|\bold{l}(c_k)-\frac{R+2\pi i}{p}|<\frac{\epsilon}{pR}$ for $k=1,2$.
\end{enumerate}
Then $\rho:\pi_1(X_p)\rightarrow PSL_2(\mathbb{C})$ is an injective map and $\rho(\pi_1(X_p))$ is a convex-cocompact subgroup of $Isom_+(\mathbb{H}^3)$.
\end{thm}

The proof of Theorem \ref{injectivity} will be deferred to Section 4. In this section, we will focus on proving Theorem \ref{main} by assuming Theorem \ref{injectivity}.

To precisely describe a viable representation $\rho:\pi_1(X_p)\rightarrow PSL_2(\mathbb{C})$ as in Theorem \ref{injectivity}, we first give parameters for such a representation. For each curve $C \in \mathcal{C'}$, it is associate with two complex numbers $\xi_C$ and $\eta_C$ such that $|\xi_C|,|\eta_C|<\epsilon$; for each $C\in \{C_1,C_2\}$, it is associate with a complex number $\xi_i$ with $|\xi_i|<\epsilon$.

We can choose parameters such that the representation $\rho:\pi_1(X_p)\rightarrow PSL_2(\mathbb{C})$ satisfies the following conditions.
\begin{enumerate}
\item For any $C\in \mathcal{C'}\cup\{C_1,C_2\}$, $\bold{hl}(C)=\frac{R}{2}+\frac{\xi_C}{R}$.
\item For any $C\in \mathcal{C'}$, $s(C)=1+\frac{\eta_C}{R^2}$.
\item $\bold{l}(c_k)=\frac{R+2\pi i}{p}+\frac{\xi_i}{pR}$ for $k=1,2$.
\end{enumerate}

Let $\mathbb{D}(0,1)$ be the disc in the complex plane centered at $0$ with radius $1$. Then for each $\tau\in \mathbb{D}(0,1)$, there is a small deformation of $\rho$, denote by $\rho_{\tau}:\pi_1(X_p)\rightarrow PSL_2(\mathbb{C})$, which is defined by the following conditions.
\begin{enumerate}
\item For any $C\in \mathcal{C'}\cup\{C_1,C_2\}$, $\bold{hl}(C)=\frac{R}{2}+\frac{\tau\xi_C}{R}$.
\item For any $C\in \mathcal{C'}$, $s(C)=1+\frac{\tau\eta_C}{R^2}$.
\item $\bold{l}(c_k)=\frac{R+2\pi i}{p}+\frac{\tau\xi_i}{pR}$ for $k=1,2$.
\end{enumerate}

Then $\{\rho_{\tau}\}_{\tau \in \mathbb{D}(0,1)}$ is a continuous family of representations from $\pi_1(X_p)$ to $PSL_2(\mathbb{C})$, such that $\rho_1=\rho$, and $\rho_0$ provides us a standard model of studying the representation $\rho:\pi_1(X_p)\rightarrow PSL_2(\mathbb{C})$.

Let $q:\widetilde{X}_p\rightarrow X_p$ be the universal cover of $X_p$. There is a natural map $\widetilde{f}_0:\widetilde{X}_p\rightarrow \mathbb{H}^3$ to realize the representation $\rho_0$. $\widetilde{f}_0$ maps each component of $\widetilde{X}_p\setminus q^{-1}(c_1\cup c_2)$ to a totally geodesic subsurface in $\mathbb{H}^3$, and two such totally geodesic subsurfaces sharing a geodesic has angle equal to $\frac{2k\pi}{p}$ for some integer $k\ne0$. For each pants $\Pi\subset X_p$, the induced map $\Pi\rightarrow \mathbb{H}^3/\rho_0(\pi_1(\Pi))$ maps $\Pi$ to a totally geodesic pants with $\bold{hl}_{\Pi}(C)=\frac{R}{2}$. 
$\widetilde{f}_0$ induces a path metric on $\widetilde{X}_p$: for any $x,y\in \widetilde{X}_p$, define $d(x,y)=inf\{l(\widetilde{f_0}(\gamma))|\ \gamma$ is a path in $\widetilde{X}_p$ with end points $x$ and $y\}$. By using elementary hyperbolic geometry, we have the following lemma.

\begin{lem}\label{qi}
For $R$ large enough, if the induced graph $G(X_p)$ satisfies $l(G(X_p))>Re^{\frac{R}{4}}$, then $\widetilde{f}_0:(\widetilde{X}_p,d)\rightarrow (\mathbb{H}^3,d_{\mathbb{H}^3})$ is injective and is a quasi-isometric embedding. In particular, $\rho: \pi_1(X_p)\rightarrow PSL_2(\mathbb{C})$ is an injective map.
\end{lem}

\begin{proof}
For any two points $x,y\in \widetilde{X}_p$, the shortest path $\alpha$ connecting $x$ and $y$ is a piecewise geodesic. Let $L_1,L_2,\cdots,L_m$ be consecutive geodesics in $q^{-1}(c_1\cup c_2)$ intersecting with $\alpha$, and let $\alpha_i$ be the segment of $\alpha$ between $L_i$ and $L_{i+1}$. Give an arbitrary orientation for each $L_i$, then the angle between $\alpha_{i-1}, L_i$ and the angle between $\alpha_i, L_i$ sum to $\pi$. Since $\bold{l}(c_k)=\frac{R+2\pi i}{p}$, the angle between $\widetilde{f}_0(\alpha_{i-1})$ and $\widetilde{f}_0(\alpha_i)$ in $\mathbb{H}^3$ is greater or equal $\frac{2\pi}{p}$.

For any $i\in\{1,\cdots,m-1\}$, $\alpha_i$ lies in a component of $\widetilde{X}_p\setminus q^{-1}(c_1\cup c_2)$ and connects two components of $q^{-1}(c_1\cup c_2)$. So $q(\alpha_i)$ is a homotopic nontrivial path in $X_p$ with end points in $c_1\cup c_2$, and it induces a path $\beta_i$ in the graph $G(X_p)$ with end points in $\{v_1,v_2\}$.

If $\beta_i$ is an essential path in $G(X_p)$, since $l(G(X_p))>Re^{\frac{R}{4}}$, the combinatorial length of $\beta_i$ is greater than $Re^{\frac{R}{4}}$. Since the distance between two different cuffs in the pair of pants with cuff length $R$ is roughly $2e^{-\frac{R}{4}}$, the length of $\alpha_i$ is greater than $R$. If $\beta_i$ is an inessential path in $G(X_p)$, then it contains a segment $e\bar{e}$ for some oriented edge $e$ of $G(X_p)$, or $\beta_i$ is just a point. Since $\alpha_i$ is a geodesic in $\widetilde{X}_p$, it contains a segment which is an essential path in the pair of pants with end points lying on the same cuff. Since the pants have cuff length $R$, such a path has length greater than $R/2$. So in this case, $\alpha_i$ has length greater than $R/2$. As a summary, for any $i\in\{1,\cdots,m-1\}$, $\alpha_i$ has length greater than $R/2$.

Now the proof reduces to an elementary exercise in hyperbolic geometry. Let $\alpha$ be a piecewise geodesic in $\mathbb{H}^3$ consists of geodesic segments $\alpha_0,\cdots,\alpha_m$, and let the length of $\alpha_i$ be $l_i$. If $l_i\geq R/2$ for each $i\in\{1,\cdots,m-1\}$, and if the angle between $\alpha_{i-1}$ and $\alpha_i$ is greater or equal to $\frac{2\pi}{p}$ for $i\in\{1,\cdots,m\}$. Then for large enough $R$ (depending on $p$), the distance between the end points of $\alpha$ in $\mathbb{H}^3$ is greater than $\frac{1}{2}\sum_{i=0}^m l_i-\frac{R}{4}$.

Since $d(x,y)=\sum_{i=0}^m l(\alpha_i)$, we have $$d_{\mathbb{H}^3}(\widetilde{f}_0(x),\widetilde{f}_0(y))\leq \sum_{i=0}^m d_{\mathbb{H}^3}(\widetilde{f}_0(x_i),\widetilde{f}_0(x_{i+1}))= \sum_{i=0}^m l(\alpha_i)= d(x,y)$$ and
$$d_{\mathbb{H}^3}(\widetilde{f}_0(x),\widetilde{f}_0(y))\geq \frac{1}{2}\sum_{i=0}^m d_{\mathbb{H}^3}(\widetilde{f}_0(x_i),\widetilde{f}_0(x_{i+1}))-\frac{R}{4}=\frac{1}{2} \sum_{i=0}^m l(\alpha_i)-\frac{R}{4}=\frac{1}{2}d(x,y)-\frac{R}{4}.$$ So $\widetilde{f}_0$ is a quasi-isometry, and $\rho$ is injective. The above inequality also implies that $\widetilde{f}_0(x)\ne \widetilde{f}_0(y)$ if $\alpha$ intersects with at least two components of $q^{-1}(c_1\cup c_2)$, and the injectivity property obviously holds for the remaining cases.
\end{proof}

We will first prove Theorem \ref{main} for finite cyclic abelian groups. As we mentioned in the introduction, we will give two proofs here. The first proof will only prove a weaker statement: the finite cyclic group embeds into the virtual homology of any closed hyperbolic $3$-manifold, but this proof is more geometric flavor. The second proof proves the original statement about virtual direct summand.

\begin{prop}\label{cyclic1}
For any finite cyclic abelian group $\mathbb{Z}_n$, and any closed hyperbolic $3$-manifold $M$, $M$ admits a finite cover $N$, such that $\mathbb{Z}_n$ embeds into $Tor(H_1(N;\mathbb{Z}))$.
\end{prop}

\begin{proof}
Let $p=2n$, then Proposition \ref{mappedcomplex} gives an immersed $2$-complex $f:X_p \looparrowright M$ with a pants decomposition $(X_p,\mathcal{C}',\{C_1,C_2\})$ such that $f_*:\pi_1(X_p)\rightarrow \pi_1(M)$ satisfies the conditions in Theorem \ref{injectivity}.

By Lemma \ref{qi}, $\rho_0(\pi_1(X_p))$ is a convex cocompact Kleinian group, so $\rho_0$ lies in \\
$int(AH(\pi_1(X_p)))$. Here $$AH(\pi_1(X_p))=\{\rho:\pi_1(X_p)\rightarrow PSL_2(\mathbb{C})|\ \rho \text{\ is\ a\ discrete,\ faithful \ representation} \}/\sim,$$
and the relation is given by conjugations.

Since the map $f:X_p \looparrowright M$ induces a viable representation $f_*:\pi_1(X_p)\rightarrow \pi_1(M)\subset PSL_2(\mathbb{C})$ satisfying the assumption of Theorem \ref{injectivity}, $f_*$ is $\pi_1$-injective.

We have pointed out that $f_*$ lies in a continuous family of viable representations $\rho_{\tau}:\pi_1(X_p)\rightarrow PSL_2(\mathbb{C})$ for $\tau\in \mathbb{D}(0,1)$, with $\rho_1=f_*$. By Theorem \ref{injectivity}, $\{\rho_{\tau}(\pi_1(X_p))\}_{\tau \in \mathbb{D}(0,1)}$ is a continuous family of convex cocompact Kleinian groups. So $f_*$ and $\rho_0$ lie in the same component of $int(AH(\pi_1(X_p)))$, and $\mathbb{H}^3/f_*(\pi_1(X_p))$ is homeomorphic to $\mathbb{H}^3/\rho_0(\pi_1(X_p))$.

Let $f_0:X_p\rightarrow \mathbb{H}^3/\rho_0(\pi_1(X_p))$ be the map induced by $\widetilde{f_0}$. Then Lemma \ref{qi} implies that $f_0$ is an embedding, and a neighborhood $\hat{M}=N(f_0(X_p))$ is a compact core of $\mathbb{H}^3/\rho_0(\pi_1(X_p))$. It is easy to figure out the topological type of $\hat{M}$.

Recall that $S'$ is an orientable surface with two oriented boundary components $C_1,C_2$ with $[C_1]-[C_2]=0\in H_1(S';\mathbb{Z})$, and $X_p$ is the quotient of $S'$. Let $V$ be the oriented solid torus and $\alpha\subset \partial V$ be the oriented $(p,1)$-curve, then $\alpha$ has a neighborhood $\alpha\times[-1,1]\subset \partial V$. Take two copies of $(V,\alpha \times [-1,1])$, and denote them by $(V_1,\alpha_1 \times [-1,1])$ and $(V_2,\alpha_2 \times [-1,1])$ respectively. Let $\psi_1: C_1\rightarrow \alpha_1$ and $\psi_2:C_2\rightarrow \alpha_2$ be two orientation preserving homeomorphisms. Let $\phi_1: C_1\times[-1,1] \rightarrow \alpha_1\times [-1,1]$ equals $\psi_1\times id$, and $\phi_2: C_2\times[-1,1] \rightarrow \alpha_2\times [-1,1]$ equals $\psi_2\times (-id)$. Then $\hat{M}=N(f_0(X_p))$ is homeomorphic to $V_1\cup_{\phi_1} S'\times[-1,1]\cup_{\phi_2} V_2$ and the boundary of $\hat{M}$ is incompressible. Such an $\hat{M}$ is called a "book of $I$-bundles" by Culler and Shalen in \cite{CS}.

Since $\mathbb{H}^3/f_*(\pi_1(X_p))$ is homeomorphic to $\mathbb{H}^3/\rho_0(\pi_1(X_p))$, $\mathbb{H}^3/f_*(\pi_1(X_p))$ has a submanifold $\hat{M}'$ homeomorphic to $\hat{M}$. It is known that fundamental groups of hyperbolic $3$-manifolds are LERF (\cite{Ag},\cite{Wi}), and by Scott's criterion of LERF (\cite{Sc}), $M$ admits an intermediate cover $N\rightarrow M$ of $\mathbb{H}^3/f_*(\pi_1(X_p))\rightarrow M$ such that $\hat{M}'$ projects to $N$ by homeomorphism.

By the description of the topological type of $\hat{M}'$, $\hat{M}'$ has only one boundary component, and $H_1(\hat{M}';\mathbb{Z})=\mathbb{Z}^{2g}\oplus \mathbb{Z} \oplus \mathbb{Z}_p$ for $g=g(S)$. The $\mathbb{Z}$-component is generated by $[c_1]$ and the $\mathbb{Z}_p$-component is generated by $[c_1]-[c_2]$. Since the image of $i_*:H_1(\partial \hat{M}';\mathbb{Z})\rightarrow H_1(\hat{M}';\mathbb{Z})$ is $\mathbb{Z}^{2g}+\mathbb{Z}[pc_1]+\mathbb{Z}[[c_1]+[c_2]]$, it is easy to show that $\mathbb{Z}_p\cap i_*(H_1(\partial \hat{M}';\mathbb{Z}))=\{0,\frac{p}{2}([c_1]-[c_2])\}$.

Then an M-V sequence argument shows that $\mathbb{Z}_{p/2}=\mathbb{Z}_n$ embeds into $H_1(N;\mathbb{Z})$.
\end{proof}

Now let's prove the statement for vitual direct summand.

\begin{prop}\label{cyclic2}
For any finite cyclic abelian group $\mathbb{Z}_n$, and any closed hyperbolic $3$-manifold $M$, $M$ admits a finite cover $N$, such that $\mathbb{Z}_n$ is a direct summand of $Tor(H_1(N;\mathbb{Z}))$.
\end{prop}

\begin{proof}
Since $M$ is a hyperbolic $3$-manifold, $\pi_1(M)$ is virtually special compact by \cite{Ag} and \cite{Wi}. So we can suppose $\pi_1(M)$ is already the fundamental group of a compact special cube complex.

Let $f:X_n\looparrowright M$ be the $\pi_1$-injective immersion we have constructed in Proposition \ref{cyclic1}. Since $f_*(\pi_1(X_n))$ is a convex-cocompact subgroup of $PSL_2(\mathbb{C})$, it is a quasi-convex subgroup of the hyperbolic group $\pi_1(M)$.

Since $f_*(\pi_1(X_n))$ is a quasi-convex subgroup of the special group $\pi_1(M)$, $f_*(\pi_1(X_n))$ is a virtual retract of $\pi_1(M)$ (\cite{HW}), i.e. $M$ admits a finite cover $N$, such that the following conditions hold.

1) $f_*(\pi_1(X_n))\subset \pi_1(N)$;

2) There exists a retraction homomorphism $r:\pi_1(N)\rightarrow \pi_1(X_n)$ such that $r\circ f_*:\pi_1(X_n)\rightarrow \pi_1(X_n)$ is identity.

So we have the induced maps on homology: $$H_1(X_n;\mathbb{Z})\xrightarrow{f_*} H_1(N;\mathbb{Z})\xrightarrow{i_*} H_1(X_n;\mathbb{Z}).$$
Since the composition is the identity map and $H_1(X_n;\mathbb{Z})=\mathbb{Z}^{2g+1} \oplus \mathbb{Z}_n$, we know that $\mathbb{Z}^{2g+1} \oplus \mathbb{Z}_n$ is a direct summand of $H_1(N;\mathbb{Z})$. In particular, $\mathbb{Z}_n$ is a direct summand of $Tor(H_1(N;\mathbb{Z}))$.
\end{proof}

\begin{rem}\label{odd}
If $n$ is an odd number, the proof of Proposition \ref{cyclic1} also shows that $\mathbb{Z}_n$ is virtually a direct summand of the homology, by taking $p=n$.
\end{rem}

\subsection{Finite Abelian Subgroups in Virtual Homology}

We will finish the proof of Theorem \ref{main} in this subsection. As in the finite cyclic group case, we will also give two proofs. One for virtual embedding, and the other one for virtual direct summand.

\begin{prop}\label{general1}
For any finite abelian group $A$, and any closed hyperbolic $3$-manifold $M$, $M$ admits a finite cover $N$, such that $A$ embeds into $Tor(H_1(N;\mathbb{Z}))$.
\end{prop}
\begin{proof}

We will proof the statement by induction on the number of generators of the finite abelian group $A$, and the proof of one-generator case has been done in Proposition \ref{cyclic1}.

For $A=\mathbb{Z}_{n_1}\oplus \mathbb{Z}_{n_2}$, take $p_k=2n_k$ for $k=1,2$. Then Proposition \ref{mappedcomplex} and Theorem \ref{injectivity} provide us two immersed $\pi_1$-injective $2$-complexes $f_1:X_{p_1}\looparrowright M$ and $f_2:X_{p_2}\looparrowright M$.

Let $G_1=(f_1)_*(\pi_1(X_{p_1}))$ and $G_2=(f_2)_*(\pi_1(X_{p_2}))$ be the two convex cocompact subgroups of $\pi_1(M)$ given by Theorem \ref{injectivity}. Let $M_k$ be the compact $3$-manifold whose interior is homeomorphic to $\mathbb{H}^3/G_k$ for $k=1,2$.

Take $g\in \pi_1(M^3)$ such that both of the two limit points of $g$ on $S^2_{\infty}$ do not lie in the limit sets $\Lambda(G_1)$ and $\Lambda(G_2)$. Then for a large enough positive integer $n$, by the Kleinian combination Theorem, $j:G_1*(g^nG_2g^{-n})\rightarrow \pi_1(M^3)$ is an embedding, and $j(G_1*(g^nG_2g^{-n}))$ is a convex cocompact subgroup of $\pi_1(M^3)$. So $\mathbb{H}^3/j(G_1*(g^nG_2g^{-n}))$ is homeomorphic to the interior of the boundary connected sum of $M_1$ and $M_2$.

Since hyperbolic $3$-manifold groups are LERF, by running the argument in the proof of Proposition \ref{cyclic1}, we can find a finite cover $N$ of $M$, such that $\mathbb{Z}_{n_1}\oplus \mathbb{Z}_{n_2}$ embeds into $H_1(N;\mathbb{Z})$.

For finite abelian groups with more generators, the result can be shown by induction as the two-generator case.
\end{proof}

Now we give the proof of Theorem \ref{main}.
\begin{proof}
We will still only show the theorem for $A=\mathbb{Z}_{n_1}\oplus\mathbb{Z}_{n_2}$ case, and the proof for general finite abelian groups follows by induction.

As in the proof of Proposition \ref{cyclic2}, we suppose $\pi_1(M)$ is already a special group. The proof of Proposition \ref{general1} provides us a quasi-convex subgroup of $\pi_1(M)$ which is isomorphic to $\pi_1(X_{n_1})*\pi_1(X_{n_2})$. Since $\pi_1(X_{n_1})*\pi_1(X_{n_2})$ is a virtual retract of $\pi_1(M)$ and $Tor\left(H_1(\pi_1(X_{n_1})*\pi_1(X_{n_2});\mathbb{Z})\right)\cong \mathbb{Z}_{n_1}\oplus \mathbb{Z}_{n_2}$, we know that $\mathbb{Z}_{n_1}\oplus \mathbb{Z}_{n_2}$ is a direct summand of $Tor(H_1(N;\mathbb{Z}))$ for some finite cover $N$ of $M$.
\end{proof}

\begin{rem}
As in Remark \ref{odd}, if the finite abelian group $A$ has odd order, the proof of Proposition \ref{general1} also shows that $A$ is a vitual direct summand.
\end{rem}

\section{The $\pi_1$-injectivily Property of Immersed $2$-Complexes}
This section is devoted to prove Theorem \ref{injectivity}. Actually, we will show a $\pi_1$-injectivity result for more general immersed $2$-complexes in closed hyperbolic $3$-manifolds, with good pants as building blocks.

For the topological pair of pants $\Pi_0$, let $\partial_k \Pi_0, k=1,2,3$ denote the three boundary components of $\Pi_0$. Suppose we have finitely many pair of pants $\mathcal{P}=\{\Pi_i\}^m_{i=1}$ and finitely many circles $\mathcal{C}=\{C_j\}_{j=1}^n$. By using these building blocks and some additional data, we will construct a $2$-complex with a "pants decomposition".

\begin{defn}\label{complex}
For each pair $(i,k),i\in\{1,\cdots,m\},k\in\{1,2,3\}$, suppose it is associated with a unique $j\in\{1,\cdots,n\}$ and a homeomorphism $f_{ik}:\partial_k \Pi_i \rightarrow C_j$. For each $C\in \mathcal{C}$, suppose it is associated with a positive integer $d_C>0$.

Let $(\cup_{i=1}^m\Pi_i)\cup(\cup_{j=1}^nC_j)\rightarrow X'$ be the quotient space quotient by the relation given by $\{\phi_{ik}\}$. Let $X$ be a further quotient space of $X'$ quotient by the $\frac{2\pi}{d_C}$-rotation on each $C\in \mathcal{C}$, and let $q:(\cup_{i=1}^m\Pi_i)\cup(\cup_{j=1}^nC_j)\rightarrow X$ be the quotient map giving $X$.
\end{defn}

For each circle $C\in \mathcal{C}$, let $D_C=d_C\cdot \#\{(k,i)|\ \partial_k \Pi_i$ is mapped to $C\}$. For any point $x\in p(C)\subset X$, a neighborhood of $x$ in $X$ is homeomorphic to the quotient space of the union of $D_C$ half-discs, by identifying their diameters together. So $D_C$ measures the the local singularity near $p(C)$.

Let $\mathcal{C}_1=\{C\in \mathcal{C}|D_C=2,d_C=1\}$ and $\mathcal{C}_2=\{C\in \mathcal{C}|D_C>2 \ \text{or} \ d_C>1\}$. If $D_C>1$ for each $C \in \mathcal{C}$ ($X$ does not have "boundary"), we say $X$ is a {\it $2$-complex with a pants decomposition $(X,\mathcal{C}_1,\mathcal{C}_2)$}.

We will call the curves in $\mathcal{C}_1$ {\it regular curves}, since each of these curves has a neighborhood in $X$ which is homeomorphic to the annulus. Curves in $\mathcal{C}_2$ are called {\it singular curves}. We can also define a graph $G(X)$ from $X$ as in step IV of our construction of $X_p$ in Section 3.1. Here vertices are given by pants in $X$, and edges are given by regular curves. In $G(X)$, all the vertices are trivalent except those vertices corresponding with pants adjacent to singular curves. $l(G(X))$ can also be defined similarly, by considering the shortest essential path in $G(X)$ with end points corresponding with pants adjacent to singular curves. Let $S(X)=\max{\{D_{C_j}\}}$, which measures the maximal singularity of $X$.

The following definition of viable representation for more general $2$-complexes is almost the same with Definition \ref{viable}.
\begin{defn}
A representation $\rho:\pi_1(X) \rightarrow PSL_2(\mathbb{C})$ is called {\it viable} if the following conditions hold.
\begin{enumerate}
\item For each $C\in \mathcal{C}_1\cup \mathcal{C}_1$, let $g_C$ be a generator of $\pi_1(C)$, then $\rho(g_C)$ is a hyperbolic element in $PSL_2(\mathbb{C})$.
\item For each pants $\Pi$ in $X$, $\rho|_{\pi_1(\Pi)}$ is an injective map, and $\rho(\pi_1(\Pi))$ is a discrete subgroup of $PSL_2(\mathbb{C})$.
\item For any two pants $\Pi, \Pi'$ sharing some $C\in \mathcal{C}_1\cup \mathcal{C}_1$, $\bold{hl}_{\Pi}(C)=\bold{hl}_{\Pi'}(C)$ holds.
\end{enumerate}
\end{defn}

For a singular curve $C\in \mathcal{C}_2$, let $N^1(\sqrt{C})$ be the half unit normal bundle of $C$ (under a viable representation $\rho$), then $N^1(\sqrt{C})\cong \mathbb{C}/\bold{hl}(C)\mathbb{Z}+2\pi i\mathbb{Z}$. There is a flow $\Psi_t$ on $N^1(\sqrt{C})$ along the direction of $\bold{hl}(C)$, and all the orbits of $\Psi_t$ are closed orbits. Suppose $\Pi_1,\cdots,\Pi_k$ are the pants adjacent to $C$, then $D_C=d_C\cdot k$. Let $foot_{\Pi_j}(C)\in N^1(\sqrt{C})$ be the foot of $\Pi_j$ on $C$.

Let $F'=\{foot_C(\Pi_j)+\frac{2l \pi i}{d_C}|\ j=1,\cdots,k,\ l=1,\cdots,d_C\}\subset N^1(\sqrt{C})$. If a point appears more than once in the definition of $F'$, we will count it with multiplicity. Let $\Psi(F')$ be the union of closed orbits passing through $F'$ under the flow $\Psi_t$. Let $F=\Psi(F')\cap \{ti|t\in\mathbb{R}/2\pi\mathbb{Z}\}\subset N^1(\sqrt{C})$, and we will consider $F$ as a subset of $S^1=\mathbb{R}/2\pi\mathbb{Z}$ (with multiplicity).

\begin{defn}
We say the pants $\Pi_1,\cdots,\Pi_n$ are {\it $p$-separated along $C$} if the distance between any two distinct points in $F$ is greater or equal to $\frac{2\pi}{p}$ on $S^1$.
\end{defn}

Then we can state our main technical theorem in this section.
\begin{thm}\label{injectivity1}
Fix an positive integer $p\geq 2$, there are universal constants $\hat{\epsilon}>0$ and $\hat{R}>0$ depending only on $p$, such that for any $0<\epsilon<\hat{\epsilon}$ and $R>\hat{R}>0$, the following statement holds. If $X$ is a connected $2$-complex with a pants decomposition $(X,\mathcal{C}_1,\mathcal{C}_2)$, and $\rho:\pi_1(X) \rightarrow PSL_2(\mathbb{C})$ is a viable representation such that:
\begin{enumerate}
\item $S(X)\leq p$ and the induced graph $G(X)$ satisfies $l(G(X))>Re^{\frac{R}{4}}$,
\item For any $C\in \mathcal{C}_1\cup \mathcal{C}_1$, it satisfies the first modified Kahn-Markovic condition:
$$      |\bold{hl}(C)-\frac{R}{2}|<\frac{\epsilon}{R},$$
\item For any $C\in \mathcal{C}_1$, it satisfies the second modified Kahn-Markovic condition:
$$      |s(C)-1|<\frac{\epsilon}{R^2}, $$
\item For any $C\in \mathcal{C}_2$, $|\bold{l}(q(C))-\frac{R+2k\pi i}{d_C}|<\frac{\epsilon}{d_CR}$ for some $k$ coprime with $d_C$, and all the pants adjacent to $C$ are $p$-separated along $C$.
\end{enumerate}
Then $\rho:\pi_1(X)\rightarrow PSL_2(\mathbb{C})$ is injective, and $\rho(\pi_1(X))$ is a convex-cocompact subgroup of $PSL_2(\mathbb{C})$.
\end{thm}

It is obviously that Theorem \ref{injectivity} is a special case of Theorem \ref{injectivity1}, so we need only to prove Theorem \ref{injectivity1} in the remaining part of this section.

We will prove Theorem \ref{injectivity1} by showing that a partially defined map $i:\widetilde{X}\rightarrow \mathbb{H}^3$ satisfying $i\circ x=\rho(x)\circ i$ for any $x\in \pi_1(X)$ is a quasi-isometric embedding. Actually, $i$ is defined on a $1$-subcomplex of $\widetilde{X}$, and $i$ maps each $1$-cell in the $1$-subcomplex to a geodesic arc in $\mathbb{H}^3$.

For each component $W$ of $\widetilde{X}\setminus p^{-1}(\mathcal{C}_2)$, we will use the work in \cite{KM1} and some classical results in hyperbolic geometry to show that $i|_W:W\rightarrow \mathbb{H}^3$ is a quasi-isometric embedding, by comparing $i|_W$ to some quasi-isometric embedding. Then we estimate the "angle" between two adjacent components of $\widetilde{X}\setminus p^{-1}(\mathcal{C}_2)$, and give it a definite lower bound. Given these estimations, we can show the globally quasi-isometric property, as the proof of Lemma \ref{qi}.

\subsection{Piecewise Linear Map is a Quasi-isometry}\label{piecewise}

To prove Theorem \ref{injectivity1}, we need to study more detail about Kahn-Markovic surfaces under the modified condition. The following Theorem in \cite{KM1} helps us to understand the boundary behavior of the universal cover of Kahn-Markovic surfaces in $\mathbb{H}^3$.

\begin{thm}\label{KMinj}
There are universal constants $\hat{\epsilon},\hat{R},K_0>0$, such that the following statement holds. For any $\epsilon,R$ such that $0<\epsilon<\hat{\epsilon}$ and $R>\hat{R}$, suppose $(S,\mathcal{C})$ is a closed surface with a pants decomposition $\mathcal{C}$, and $f:S\looparrowright M$ is a viable map such that
$$
\left\{ \begin{array}{l}
        |\bold{hl}(C)-\frac{R}{2}|<\epsilon, \\
        |s(C)-1|<\frac{\epsilon}{R}
\end{array} \right.
$$ holds for any $C\in\mathcal{C}$. Then $\rho=f_*:\pi_1(S)\rightarrow PSL_2(\mathbb{C})$ is injective and $\rho(\pi_1(S))$ is a quasi-fuchsian group.

Moreover, suppose $S$ is endowed with the hyperbolic metric with $\bold{hl}(C)=\frac{R}{2}$ and $s(C)=1$, then $\partial \widetilde{f}:\partial \widetilde{S}=\partial\mathbb{H}^2 \rightarrow \partial \widetilde{M^3}=\partial \mathbb{H}^3$ extends to a $\pi_1(S)$-equivariant $(1+K_0\epsilon)$-quasiconformal map $\partial \widetilde{g}: \partial \mathbb{H}^3 \rightarrow \partial \mathbb{H}^3$ (by considering $\pi_1(S)\subset Isom_+(\mathbb{H}^2)\subset Isom_+(\mathbb{H}^2)$).
\end{thm}

Since we can suppose that $f$ satisfies the modified Kahn-Markovic condition:
$$
\left\{ \begin{array}{l}
        |\bold{hl}(C)-\frac{R}{2}|<\frac{\epsilon}{R}, \\
        |s(C)-1|<\frac{\epsilon}{R^2},
\end{array} \right.
$$
the induced map $\partial \widetilde{f}:\partial \widetilde{S}=\partial\mathbb{H}^2 \rightarrow \partial \widetilde{M^3}=\partial \mathbb{H}^3$ extends to a $\pi_1(S)$-equivariant $(1+\frac{K_0\epsilon}{R})$-quasiconformal map $\partial \widetilde{g}:\partial \mathbb{H}^3\rightarrow \partial \mathbb{H}^3$.

Let $\rho_0: \pi_1(S)\rightarrow PSL_2(\mathbb{C})$ be the representation near $\rho$ satisfying $\bold{hl}(C)=\frac{R}{2}, s(C)=1$ for any $C\in \mathcal{C}$, then the map $\partial \widetilde{g}:S^2_{\infty}\rightarrow S^2_{\infty}$ is a $(1+\frac{K_0\epsilon}{R})$-quasiconformal conjugacy between $\rho_0(\pi_1(S))$ and $\rho(\pi_1(S))$.

Such a quasiconformal conjugacy between two Kleinian groups can be extended to a quasi-isometry from $\mathbb{H}^3/\rho_0(\pi_1(S))$ to $\mathbb{H}^3/\rho(\pi_1(S))$, and a quantitative version of this result is proved in \cite{Th1} Chapter 11 and \cite{McM} Corollary B.23.

\begin{thm}\label{McM}
Let $M_i=\mathbb{H}^3/\Gamma_i,\ i=1,2$ be two hyperbolic $3$-manifolds with isomorphic fundamental groups, and let $\phi:S^2_{\infty}\rightarrow S^2_{\infty}$ be a $K$-quasiconformal conjugation between $\Gamma_1$ and $\Gamma_2$. Then $\phi$ extends to an equivariant $K^{\frac{3}{2}}$-quasi-isometry $\Phi:\mathbb{H}^3\rightarrow \mathbb{H}^3$. In particular, $M_1$ and $M_2$ are $K^{\frac{3}{2}}$-quasi-isometric with each other.
\end{thm}

The proof of Theorem \ref{McM} is quite differential geometry style, and the $L$-quasi-isometry showed up in the statement means $(L,0)$-quasi-isometry in the coarse geometry sense. In the following part of this paper, we will use the language in the coarse geometry.

Given the previous a few theorems and the modified Kahn-Markovic condition, we have the following quick corollary.

\begin{col}\label{quasiisometry}
For any closed hyperbolic $3$-manifold $M$, there exist constants $K_0, \hat{\epsilon}>0, \hat{R}>0$ such that the following statement hold. Suppose $S$ is a hyperbolic surface with a pants decomposition $\mathcal{C}$, and $f:S\looparrowright M$ is a viable map satisfying the modified Kahn-Markovic condition
$$
\left\{ \begin{array}{l}
        |\bold{hl}(C)-\frac{R}{2}|<\frac{\epsilon}{R}, \\
        |s(C)-1|<\frac{\epsilon}{R^2},
\end{array} \right.
$$
for each $C\in \mathcal{C}$, with $0<\epsilon<\hat{\epsilon}$ and $R>\hat{R}$.

Let $\rho=f_*:\pi_1(S)\rightarrow PSL_2(\mathbb{C})$ be the induced map on the fundamental group, and $\rho_0:\pi_1(S)\rightarrow PSL_2(\mathbb{C})$ be the representation near $\rho$ and satisfying $\bold{hl(C)}=\frac{R}{2},\ s(C)=1$.

Then there exists a $(1+\frac{2K_0\epsilon}{R},0)$-quasi-isometry $\widetilde{g}:\mathbb{H}^3 \rightarrow \mathbb{H}^3$ such that $\widetilde{g} \circ \rho_0(x)=\rho(x)\circ\widetilde{g}$ for any $x\in \pi_1(S)$. In particular, there exists a $(1+\frac{2K_0\epsilon}{R},0)$-quasi-isometry $g:M_1 \rightarrow M_2$.
\end{col}

Although we know the existence of the equivariant quasi-isometry $\widetilde{g}:\mathbb{H}^3\rightarrow \mathbb{H}^3$, we need to know more detail about it. So we will use a more concrete partially defined map to approximate $\widetilde{g}$.

Let $(S,\mathcal{C})$ be the pants decomposition in Corollary \ref{quasiisometry}. Endow $S$ with the hyperbolic metric with $\bold{hl}(C)=\frac{R}{2}$ and $s(C)=1$ for any $C\in \mathcal{C}$, then $\mathcal{C}$ is a union of simple closed geodesics on $S$.

Let $\widetilde{\mathcal{C}} \subset \mathbb{H}^2$ be the preimage of $\mathcal{C}$. For each pair of pants $\Pi$ in $S$ with three cuffs $c_i,\ i=1,2,3$, there are three seams $a_i,\ i=1,2,3$ in $\Pi$ with $a_i$ connects $c_{i-1}$ and $c_{i+1}$. We suppose that $a_i$ is perpendicular with both $c_{i-1}$ and $c_{i+1}$ in the hyperbolic surface $S$. Let $\mathcal{A}$ be the union of all the seams in $S$, and $\widetilde{\mathcal{A}} \subset \mathbb{H}^2$ be the preimage of $\mathcal{A}$. Let $Y=\widetilde{\mathcal{C}}\cup \widetilde{\mathcal{A}}\subset \mathbb{H}^2\subset \mathbb{H}^3$, then $Y$ is a $\rho_0(\pi_1(S))$-equivariant $1$-subcomplex in $\mathbb{H}^3$, and $Y \subset \mathbb{H}^2$ gives a tessellation of $\mathbb{H}^2$ by isometric right-angled hexagons.

In Corollary \ref{quasiisometry}, we can suppose $\widetilde{g}:\mathbb{H}^3 \rightarrow \mathbb{H}^3$ has been extended to a self-map on $\mathbb{H}^3\cup S^2_{\infty}$ and we still denote it by $\widetilde{g}$.

Let $Y'$ be the $\rho(\pi_1(S))$-equivariant $1$-complex in $\mathbb{H}^3$ defined as the following. For any geodesic $\gamma \subset \widetilde{\mathcal{C}}$ with end points $x,y \in S^2_{\infty}$, we use $\partial \widetilde{g}(\gamma)$ to denote the geodesic with end points $\widetilde{g}(x)$ and $\widetilde{g}(y)$. For any geodesic arc $\alpha \subset \widetilde{\mathcal{A}}$ orthogonal with geodesics $\gamma_1,\gamma_2 \subset \widetilde{\mathcal{C}}$, we use $\partial \widetilde{g}(\alpha)$ to denote the common perpendicular geodesic of $\partial \widetilde{g}(\gamma_1)$ and $\partial \widetilde{g}(\gamma_2)$. Then we define $Y'=\partial \widetilde{g}(\widetilde{\mathcal{C}}) \cup \partial \widetilde{g}(\widetilde{\mathcal{A}})$.

Now we can define a piecewise linear $\pi_1(S)$-equivariant homeomorphism $h:Y\rightarrow Y'$. $h$ maps each $\alpha \subset \widetilde{\mathcal{A}}$ to $\partial \widetilde{g}(\alpha)$ linearly, and $h$ maps each $\gamma \subset \widetilde{\mathcal{C}}$ to $\partial \widetilde{g}(\gamma)$ piecewise linearly such that the restriction of $h$ on each component of $\gamma\setminus \widetilde{\mathcal{A}}$ is linear.

Since the modified Kahn-Markovic condition is satisfied by $\rho$, it is easy to check that $h$ is a $(1+\frac{2\epsilon}{R},0)$-quasi-isometry on each geodesic segment of $Y=\widetilde{\mathcal{C}}\cup \widetilde{\mathcal{A}}$. However, we do not know whether $h:Y\rightarrow Y'$ is a globally quasi-isometry yet. By using the information given by $\widetilde{g}$, we will show that $h$ is actually a quasi-isometry.

\begin{thm}\label{quasi}
Under the induced metric from $\mathbb{H}^3$, $h:Y\rightarrow Y'$ is a $(1+\frac{K\epsilon}{R},K(\epsilon+\frac{1}{R})^{\frac{1}{5}})$-quasi-isometry for some universal constant $K>0$.
\end{thm}

We need to show the following elementary lemma first, such kind of statements are well-known and we give a quantitative version here.

\begin{lem}\label{delta}
For small enough $\delta>0$, suppose $\gamma:[0,n]\rightarrow \mathbb{H}^3$ is a $(1+\delta,\delta)$-quasigeodesic in $\mathbb{H}^3$, and let $\bar{\gamma}$ be the geodesic with end points $\gamma(0)$ and $\gamma(n)$, then $\gamma\subset N_{\eta}(\bar{\gamma})$ for $\eta=5\delta^{\frac{1}{5}}$.
\end{lem}

\begin{proof}
Suppose $\gamma(t_0)$ is the point on $\gamma([0,n])$ which is the farthest one from $\bar{\gamma}$, and let $D=d(\gamma(t_0),\bar{\gamma})$. We can suppose $D>2\delta^{\frac{1}{5}}$, or the lemma holds trivially. In this case, $t_0,n-t_0>\frac{2\delta^{\frac{1}{5}}-\delta}{1+\delta}>\delta^{\frac{1}{5}}$.

Let $t_1=t_0-\delta^{\frac{1}{5}}$, $t_2=t_0+\delta^{\frac{1}{5}}$, let $l_i=d(\gamma(t_0),\gamma(t_i))$ for $i=1,2$, and $d_i=d(\gamma(t_i),\bar{\gamma})\leq D$. We use $x_j$ to denote the orthogonal projection of $\gamma(t_j)$ on $\bar{\gamma}$ for $j=0,1,2$.

For two points $x,y \in \mathbb{H}^3$, we will use $\overline{xy}$ to denote the geodesic segment with end points $x$ and $y$.

Let $\phi_i$ be the angle difference between $\overline{\gamma(t_0)x_0}$ and $\overline{\gamma(t_i)x_i}$ along geodesic $\bar{\gamma}$ for $i=1,2$ (by using the parallel transportation from $x_0$ to $x_i$). Let $\theta_i$ be the angle between $\overline{\gamma(t_i)\gamma(t_0)}$ and $\overline{\gamma(t_0)x_0}$, and let $\theta$ be the angle between $\overline{\gamma(t_1)\gamma(t_0)}$ and $\overline{\gamma(t_2)\gamma(t_0)}$. Then $\theta\leq \theta_1+\theta_2$.

A computation in hyperbolic geometry gives:
\begin{equation}\label{1}
\cos{\theta_i}=\frac{\sinh{D}\cosh{l_i}-\sinh{d_i}\cos{\phi_i}}{\cosh{D}\sinh{l_i}}\geq \frac{\sinh{D}\cosh{l_i}-\sinh{D}}{\cosh{D}\sinh{l_i}}=\tanh{D}\tanh{\frac{l_i}{2}}.
\end{equation}

Since $t_2-t_1=2\delta^{\frac{1}{5}}$, and $\gamma$ is a $(1+\delta,\delta)$-quasi-isometry, $d(\gamma(t_1),\gamma(t_2))\geq \frac{2\delta^{\frac{1}{5}}}{1+\delta}-\delta$.

By the hyperbolic cosine law and \eqref{1}:
\begin{equation}\label{2}
\begin{split}
\cosh{d(\gamma(t_1),\gamma(t_2))}& =\cosh{l_1}\cosh{l_2}-\sinh{l_1}\sinh{l_2}\cos{\theta}\\
& \leq \cosh{l_1}\cosh{l_2}-\sinh{l_1}\sinh{l_2}\cos{(\theta_1+\theta_2)} \\
& \leq \cosh{l_1}\cosh{l_2}+\sinh{l_1}\sinh{l_2}(1-\tanh^2{D}\tanh{\frac{l_1}{2}}\tanh{\frac{l_2}{2}}).
\end{split}
\end{equation}

So we have
\begin{equation}
\tanh^2{D}\leq\frac{\cosh{(l_1+l_2)}-\cosh{d(\gamma(t_1),\gamma(t_2))}}{\sinh{l_1}\sinh{l_2}\tanh{\frac{l_1}{2}}\tanh{\frac{l_2}{2}}}.
\end{equation}

Since $\frac{\delta^{\frac{1}{5}}}{1+\delta}-\delta\leq l_1, l_2\leq (1+\delta)\delta^{\frac{1}{5}}+\delta$ and $d(\gamma(t_1),\gamma(t_2))\geq \frac{2\delta^{\frac{1}{5}}}{1+\delta}-\delta$, we have:
\begin{equation}
\tanh^2{D}\leq \frac{\cosh{(2(1+\delta)\delta^{\frac{1}{5}}+2\delta)}-\cosh{(\frac{2\delta^{\frac{1}{5}}}{1+\delta}-\delta)}}{\sinh^2{(\frac{\delta^{\frac{1}{5}}}{1+\delta}-\delta)}
\tanh^2{(\frac{\delta^{\frac{1}{5}}}{2(1+\delta)}-\frac{\delta}{2})}}=24\delta^{\frac{2}{5}}(1+O(\delta^{\frac{1}{5}})).
\end{equation}

So $D=\sqrt{24}\delta^{\frac{1}{5}}(1+O(\delta^{\frac{1}{5}}))<5\delta^{\frac{1}{5}}$.
\end{proof}

Since $\widetilde{g}:\mathbb{H}^3 \rightarrow \mathbb{H}^3$ is a $(1+\frac{2K_0\epsilon}{R},0)$-quasi-isometry, by Lemma \ref{delta}, we know that $\widetilde{g}(\widetilde{\mathcal{C}})\subset N_{\eta}(\partial \widetilde{g}(\widetilde{\mathcal{C}}))$ for $\eta=5(\frac{2K_0\epsilon}{R})^{\frac{1}{5}}$. Let $p:\widetilde{g}(\widetilde{\mathcal{C}})\rightarrow \partial \widetilde{g}(\widetilde{\mathcal{C}})$ be the nearest point projection from $\widetilde{g}(\gamma)$ to $\partial \widetilde{g}(\gamma)$ for each $\gamma \subset \widetilde{\mathcal{C}}$. Since $p$ moves every point in $\widetilde{g}(\widetilde{\mathcal{C}})$ by at most $\eta$. Let $\widetilde{g}'=p\circ \widetilde{g}|_{\widetilde{\mathcal{C}}}:\widetilde{\mathcal{C}} \rightarrow \partial \widetilde{g}(\widetilde{\mathcal{C}})$, then $\widetilde{g}'$ is a $\pi_1(S)$-equivariant $(1+\frac{2K_0\epsilon}{R},2\eta)$-quasi-isometry. We will compare $\widetilde{g}'$ and $h$, which shows that $h$ is a quasi-isometry.

Now we are ready to prove Theorem \ref{quasi}.

\begin{proof}
For an oriented geodesic $\gamma \in \widetilde{\mathcal{C}}$ corresponding with a curve $C\subset \mathcal{C}$, we will also use $\gamma$ to denote the hyperbolic isometry corresponding with the oriented curve $C\subset S$.

For any $\alpha \in \widetilde{\mathcal{A}}$ which is orthogonal to $\gamma_1,\gamma_2 \in \widetilde{\mathcal{C}}$, we give orientations for $\alpha, \gamma_1,\gamma_2$ as in Figure 1. Let $x=\alpha\cap \gamma_2$, $y_1=\gamma_2(x)$ and $y_2=\gamma_1(y_1)$.

\begin{center}
\psfrag{a}[]{$\gamma_1$} \psfrag{b}[]{$\gamma_2$} \psfrag{c}[]{$\gamma_1(\gamma_2)$} \psfrag{d}[]{$\alpha$}
\psfrag{x}[]{$x$} \psfrag{y}[]{$y_1$} \psfrag{z}[]{$y_2$}
\includegraphics[width=3.5in]{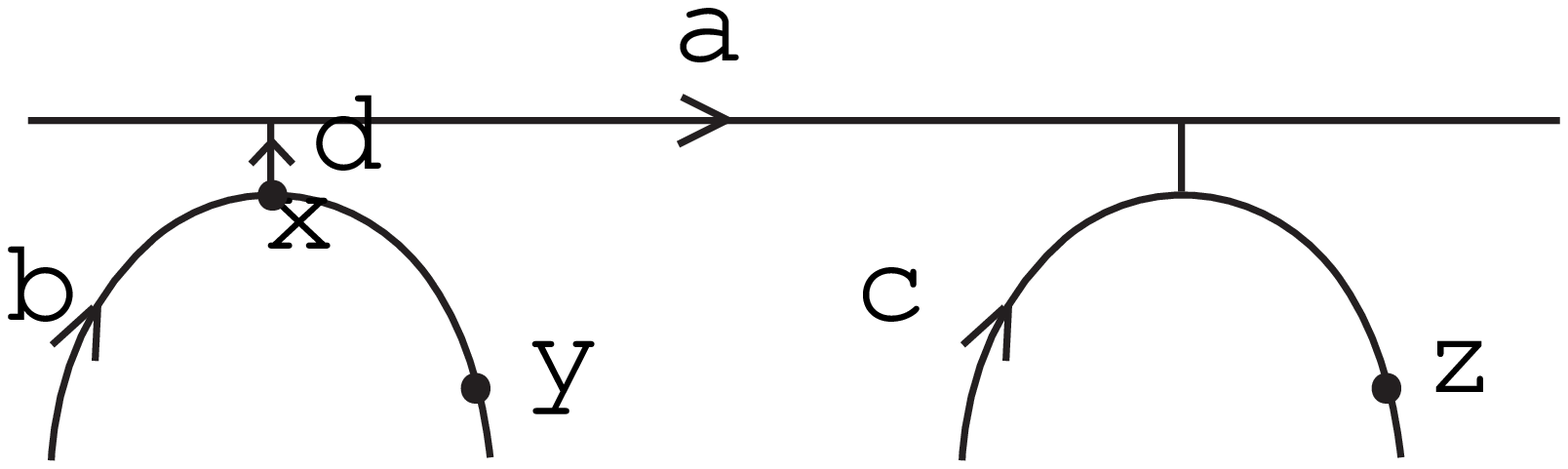}
\vskip 0.5 truecm
 \centerline{Figure 1}
\end{center}

We will compare $\widetilde{g}'(x)$ and $h(x)$ on $\partial \widetilde{g}(\gamma_2)$, and show that they are close to each other.

Let $d_1=d(\gamma_1,\gamma_2)$. By the hyperbolic cosine law of right-angled hexagons, and all the curves $C\subset \mathcal{C}$ satisfy $\bold{hl}(C)=\frac{R}{2}$, we have $\cosh{d_1}=\frac{\cosh{\frac{R}{2}}}{\cosh{\frac{R}{2}}-1}$. Let $d_2=d(y_1,\gamma_1)$. Since $d(x,y)=R$, we have $\sinh{d_2}=\sinh{d_1}\cosh{R}$, and
\begin{equation}
\sinh^2{d_2}=\cosh^2{R}\frac{2\cosh{\frac{R}{2}}-1}{(\cosh{\frac{R}{2}}-1)^2}.
\end{equation}

By computations in hyperbolic geometry of $\mathbb{H}^2$, we have
\begin{equation}
\begin{split}
\cosh{d(y_1,y_2)} &=\cosh^2{d_2}\cosh{R}-\sinh^2{d_2}\\
& =\cosh{R}+\cosh^2{R}(\cosh{R}-1)\frac{2\cosh{\frac{R}{2}}-1}{(\cosh{\frac{R}{2}}-1)^2}\\
& =\frac{1}{2}e^{\frac{5}{2}R}(1+O(e^{-\frac{R}{2}})).
\end{split}
\end{equation}
So $d(y_1,y_2)=\frac{5}{2}R+O(e^{-\frac{R}{2}})$.

Now we think about the position of $\widetilde{g}'(x)$ on $\partial \widetilde{g}(\gamma_2)$. Let $l$ be the oriented distance between $h(x)=h(\gamma_2)\cap h(\alpha)$ and $\widetilde{g}'(x)$ on $h(\gamma_2)$. We will prove that $l$ is very small.

Let $d_1'=\bold{d}_{\alpha}(\gamma_2,\gamma_1)$, by the hyperbolic cosine law of right-angled hexagons,
\begin{equation}\label{7}
d_1'=2e^{-\frac{R}{4}+\frac{1}{2R}(\epsilon_3-\epsilon_2-\epsilon_1)}+O(e^{-\frac{3R}{4}})
\end{equation}
for complex numbers $\epsilon_i$ with $|\epsilon_i|<\epsilon$ for $i=1,2,3$. Let $d_2'=d(\widetilde{g}'(y_1),\partial \widetilde{g}(\gamma_1))$, and let the real length on $\partial \widetilde{g}(\gamma_2)$ be $R+\frac{2\epsilon_4}{R}$ for some real number $\epsilon_4$ with $|\epsilon_4|<\epsilon$, then
\begin{equation}\label{8}
\begin{split}
\sinh^2{d_2'}& =\cosh^2{(l+R+\frac{2\epsilon_4}{R})}\sinh^2{(\Re{(d_1')})}+\sinh^2{(l+R+\frac{2\epsilon_4}{R})}\sin^2{(\Im{(d_1')})}\\
& \geq \cosh^2{(l+R-\frac{2\epsilon}{R})\sinh^2{\left((2+O(e^{-\frac{R}{2}}))e^{-\frac{R}{4}-\frac{3\epsilon}{2R}}\right)}}.
\end{split}
\end{equation}
Here $\Re{(z)}$ and $\Im{(z)}$ are the real and imaginary part of a complex number $z$ respectively.

Let $R_1=\bold{l}(h(\gamma_1))$ be the complex translation length of $h(\gamma_1)$, then $|R_1-R|<\frac{2\epsilon}{R}$. So by \eqref{7} and \eqref{8}, we have
\begin{equation}
\begin{split}
& \cosh{(d(g'(y_1),g'(y_2)))}\\
 = & \cosh^2{d_2'}\cosh{\Re{(R_1)}}-\sinh^2{d_2'}\cos{\Im{(R_1)}}\\
 \geq & \cosh^2{d_2'}\cosh{(R-\frac{2\epsilon}{R})}-\sinh^2{d_2'}\\
 = & \cosh{(R-\frac{2\epsilon}{R})}+(\cosh{(R-\frac{2\epsilon}{R})}-1)\sinh^2{d_2'}\\
\geq & \cosh{(R-\frac{2\epsilon}{R})}+(\cosh{(R-\frac{2\epsilon}{R})}-1)\cosh^2{(l+R-\frac{2\epsilon}{R})}\sinh^2{\left((2+O(e^{-\frac{R}{2}}))e^{-\frac{R}{4}-\frac{3\epsilon}{2R}}\right)}\\
= &\frac{1}{2}e^{2l+\frac{5R}{2}-\frac{9\epsilon}{R}}(1+O(e^{-\frac{R}{2}})).
\end{split}
\end{equation}

So $d(\widetilde{g}'(y_1),\widetilde{g}'(y_2))\geq |2l+\frac{5R}{2}-\frac{9\epsilon}{R}+O(e^{-\frac{R}{2}})|.$ Since $\widetilde{g}'$ is a $(1+\frac{2K_0\epsilon}{R},2\eta)$-quasi-isometry and $d(y_1,y_2)=\frac{5}{2}R+O(e^{-\frac{R}{2}})$, the following inequality holds:
\begin{equation}\label{10}
(1+\frac{2K_0\epsilon}{R})(\frac{5}{2}R+O(e^{-\frac{R}{2}}))+2\eta\geq |2l+\frac{5R}{2}-\frac{9\epsilon}{R}+O(e^{-\frac{R}{2}})|.
\end{equation}

By considering $y_1'=\gamma_2^{-1}(x)$ and $y_2'=\gamma_1^{-1}(y_1)$ and run the same argument, we have
\begin{equation}\label{11}
(1+\frac{2K_0\epsilon}{R})(\frac{5}{2}R+O(e^{-\frac{R}{2}}))+2\eta \geq |-2l+\frac{5R}{2}-\frac{9\epsilon}{R}+O(e^{-\frac{R}{2}})|.
\end{equation}

By \eqref{10} and \eqref{11} and $\eta=5(\frac{2K\epsilon}{R})^{\frac{1}{5}}$, we get
\begin{equation}
|l|\leq \frac{5}{2}K_0 \epsilon+\frac{9\epsilon}{2R}+\eta+O(e^{-\frac{R}{2}})\leq K_1(\epsilon+\frac{1}{R}+(\frac{\epsilon}{R})^{\frac{1}{5}}).
\end{equation}

Since $l$ is the oriented distance between $\widetilde{g}'(x)$ and $h(x)$, $d(\widetilde{g}'(x),h(x))\leq K_1(\epsilon+\frac{1}{R}+(\frac{\epsilon}{R})^{\frac{1}{5}})$ holds for any $x \in \widetilde{\mathcal{C}} \cap \widetilde{\mathcal{A}}$.

Since $\widetilde{g}':\widetilde{\mathcal{C}}\rightarrow \partial \widetilde{g}(\widetilde{\mathcal{C}})$ is a $(1+\frac{2K_0\epsilon}{R},2\eta)$-quasi-isometry and the restriction of  $h:\widetilde{\mathcal{C}}\rightarrow \partial \widetilde{g}(\widetilde{\mathcal{C}})$ on each single geodesic $C \subset \widetilde{\mathcal{C}}$ is a $(1+\frac{2\epsilon}{K},0)$-quasi-isometry. So $d(\widetilde{g}'(y),h(y))\leq 2\epsilon+2K_0\epsilon+2\eta+K_1(\epsilon+\frac{1}{R}+(\frac{\epsilon}{R})^{\frac{1}{5}})$ for each $y\in \widetilde{\mathcal{C}}$.

So $h|_{\widetilde{\mathcal{C}}}:\widetilde{\mathcal{C}}\rightarrow \partial \widetilde{g}(\widetilde{\mathcal{C}})$ is a $\pi_1(S)$-equivariant $(1+\frac{2K_0\epsilon}{R},K_2(\epsilon+\frac{1}{R}+(\frac{\epsilon}{R})^{\frac{1}{5}}))$-quasi-isometry. Since ${\widetilde{\mathcal{C}}}\subset Y$ is $2e^{-\frac{R}{4}}$-dense, $h:Y\rightarrow Y'$ is a $(1+\frac{2K_0\epsilon}{R},K_3(\epsilon+\frac{1}{R}+(\frac{\epsilon}{R})^{\frac{1}{5}}))$-quasi-isometry. \end{proof}

\subsection{Estimation of the Angle Change}

Now we consider two points $x,y\in Y$ with $x$ lying on a geodesic $\gamma \subset \widetilde{\mathcal{C}}$ and $d(x,y)\geq \frac{R}{2}$. We give $\gamma$ an arbitrary orientation. Let $\alpha \subset \widetilde{\mathcal{A}}$ be the geodesic arc which is at the same side of $\gamma$ as $\overline{xy}$ on $\mathbb{H}^2$, such that it is the closest such geodesic arc to $x$ (it is possible there are two choices and we choose either of them). Let $\vec{e}_1 \in T^1_x(\mathbb{H}^3)$ be the unit vector in the direction of $\gamma$, $\vec{e}_2 \in T^1_x(\mathbb{H}^3)$ be the unit vector in the direction of the parallel transportation of $\alpha$ to $x$, and take the third unit vector $\vec{e}_3 \in T^1_x(\mathbb{H}^3)$ such that $(\vec{e}_1,\vec{e}_2,\vec{e}_3)$ is an orthonormal frame of $T_x(\mathbb{H}^3)$ and gives the orientation of $\mathbb{H}^3$. Let $\vec{e}$ be the unit vector in $T^1_x(\mathbb{H}^3)$ in the direction of $\overline{xy}$.

Now we define $\Theta{(\gamma,\alpha,\overline{xy})}=(\theta,\phi) \in \mathbb{R}^2$ for $\theta=\angle{(\vec{e},\vec{e}_1)}\in [0,\pi]$ and $\phi=\angle{(\vec{e},\vec{e}_3)}\in [0,\pi]$. $\Theta{(\gamma,\alpha,\overline{xy})}$ measures the direction of $\vec{e}$ under the coordinate $(\vec{e}_1,\vec{e}_2,\vec{e}_3)$. Since $\rho_0(\pi_1(S))$ is a Fuchsian group, $\phi=\frac{\pi}{2}$. $\Theta{(h(\gamma),h(\alpha),\overline{h(x)h(y)})}=(\theta',\phi')\in \mathbb{R}^2$ is defined similarly.

The main result of this subsection is the following statement.

\begin{thm}\label{angle}
There exists constants $\hat{\epsilon}>0$, $\hat{R}>0$ depend only on $p$, such that for any $0<\epsilon<\hat{\epsilon}$ and $R>\hat{R}$, $|\Theta{(\gamma,\alpha,\overline{xy})}-\Theta{(h(\gamma),h(\alpha),\overline{h(x)h(y)})}|\leq \frac{1}{p}$.
\end{thm}

Since each hexagon in $\mathbb{H}^2\setminus Y$ has diameter $\leq \frac{R}{2}+1$, $\overline{xy}\cap Y=\{x_0,x_1,\cdots,x_n\}$ for $x=x_0$, $y=x_n$ and $d(x_i,x_{i+1})\leq \frac{R}{2}+1$. Let $h(\overline{xy})$ denote the piecewise geodesic in $\mathbb{H}^3$ which is the concatenation of $\overline{h(x_i)h(x_{i+1})}$. There is a natural piecewise linear map $h':\overline{xy}\rightarrow h(\overline{xy})$ such that $h'(x_i)=h(x_i)$. Since $h:Y \rightarrow Y'$ is a $(1+\mu_1,\mu_2)$-quasi-isometry for $\mu_1=\frac{K\epsilon}{R}$, and $d(x_i,x_{i+1})\leq \frac{R}{2}+1$, it is easy to check that $h'$ is a $(1+\mu_1,3\mu_2+4K\epsilon)$-quasi-isometry. Since $\mu_2=K(\epsilon+\frac{1}{R})^{\frac{1}{5}}$, $h'$ is a $(1+\mu_1,10\mu_2)$-quasi-isometry in particular.

Let $x_k$ be the point in $\{x_0,\cdots,x_n\}$ which is the nearest one to $x$ and such that
\begin{equation}\label{lp}
l(p) \leq d(x,x_k)\leq R
\end{equation}
holds for $l(p)=\frac{1}{10000p^2}$.
If $x_k\in \alpha \subset \widetilde{\mathcal{A}}$, let $x'$ be one of the intersection points of $\alpha \cap \widetilde{\mathcal{C}}$. If $x_k$ does not lie in $\widetilde{\mathcal{A}}$, simply let $x'=x_k$. By Lemma \ref{delta} and the choice of $x'$, we have $d(h(x'),\overline{h(x)h(y)})\leq 5(\mu_2)^{\frac{1}{5}}+2e^{-\frac{R}{4}}\leq 6(\mu_2)^{\frac{1}{5}}$. Since $d(h(x),h(x'))\geq \frac{l(p)}{1+\mu_1}-\mu_2- 2e^{-\frac{R}{4}}\geq \frac{l(p)}{2}$, by the hyperbolic sine law,
\begin{equation}
\sin{(\angle{h(x')h(x)h(y)})}=\frac{\sinh{(d(h(x'),\overline{h(x)h(y)}))}}{\sinh{(d(h(x),h(x')))}}\leq \frac{\sinh{6(\mu_2)^{\frac{1}{5}}}}{\sinh{\frac{l(p)}{2}}}\leq \frac{15(\mu_2)^{\frac{1}{5}}}{l(p)}.
\end{equation}

So $\angle{h(x')h(x)h(y)}\leq \frac{20(\mu_2)^{\frac{1}{5}}}{l(p)}$, and the same argument shows $\angle{x'xy}\leq \frac{20(\mu_2)^{\frac{1}{5}}}{l(p)}$. Now it suffices to show that $|\Theta{(\gamma,\alpha,\overline{xx'})}-\Theta{(h(\gamma),h(\alpha),\overline{h(x)h(x')})}|\leq \frac{1}{2p}$. Even if $x'$ may not be same with $x_k$, we will abuse the notation and still use $x_0=x,x_1,\cdots,x_k=x'$ to denote the intersection points in $\overline{xx'}\cap Y$. Moreover, we still use the notation $\Theta{(\gamma,\alpha,\overline{xx'})}=(\theta,\phi)$ and $\Theta{(h(\gamma),h(\alpha),\overline{h(x)h(x')})}=(\theta',\phi')$, with $\phi=\frac{\pi}{2}$.

\begin{prop}\label{theta}
Let $\theta$ and $\theta'$ be the first coordinate of $\Theta{(\gamma,\alpha,\overline{xx'})}$ and \\$\Theta{(h(\gamma),h(\alpha),\overline{h(x)h(x')})}$ respectively, then $|\theta-\theta'|\leq\frac{1}{4p}$.
\end{prop}

\begin{proof}
Let $d_1=d(x,x')$, $d_2=d(x',\gamma)$ and $d_1'=d(h(x),h(x'))$, $d_2'=d(h(x'),h(\gamma))$, then $d_2\leq d_1$ and $d_2'\leq d_1'$. Since $h:Y\rightarrow Y'$ is a $(1+\mu_1,\mu_2)$-quasi-isometry, and $l(p)\leq d_1 \leq R$, we have $|d_1'-d_1|\leq K\epsilon+\mu_2$ and $|d_2'-d_2|\leq K\epsilon+\mu_2$. In particular, $d_1'\geq \frac{1}{2}l(p)$.

Since $\sin{\theta}=\frac{\sinh{d_2}}{\sinh{d_1}}$ and $\sin{\theta'}=\frac{\sinh{d_2'}}{\sinh{d_1'}}$, we have the following estimation:
\begin{equation}
\begin{split}
& |\sin{\theta}-\sin{\theta'}|\\
\leq & \frac{\sinh{d_2}\cdot|\sinh{d_1}-\sinh{d_1'}|+\sinh{d_1}\cdot|\sinh{d_2}-\sinh{d_2'}|}{\sinh{d_1}\sinh{d_1'}}\\
\leq & \frac{|d_1-d_1'|\cdot \cosh{(\max{(d_1,d_1')})}+|d_2-d_2'|\cdot \cosh{(\max{(d_2,d_2')})}}{\sinh{d_1'}} \\
\leq & 2(K\epsilon+\mu_2)\frac{\cosh{(\max{(d_1,d_1')})}}{\sinh{d_1'}}\\
\leq & 2(K\epsilon+\mu_2)e^{K\epsilon+\mu_2}\coth{d_1'}\\
\leq & 2(K\epsilon+\mu_2)e^{K\epsilon+\mu_2}\coth{(\frac{1}{2}l(p))}.
\end{split}
\end{equation}

Let $\nu=2(K\epsilon+\mu_2)e^{K\epsilon+\mu_2}\coth{(\frac{1}{2}l(p))}$. If both $\theta$ and $\theta'$ are acute angles, then
\begin{equation}
\nu\geq |\sin{\theta}-\sin{\theta'}| =2\sin{\frac{|\theta-\theta'|}{2}}\cos{\frac{\theta+\theta'}{2}}\geq 2\sin^2{\frac{|\theta-\theta'|}{2}}\geq \frac{(\theta-\theta')^2}{8}.
\end{equation}
So in this case $|\theta-\theta'|\leq\sqrt{8\nu}$.

Without lose of generality, we can suppose that $\theta$ is an acute angle. If $\theta'$ is also acute, the above inequality gives $|\theta-\theta'|\leq\sqrt{8\nu}\leq \frac{1}{4p}$, so the lemma is proved.

If $\theta'$ is not acute, we will show that $\theta$ is very close to $\pi/2$. Let $\beta$ be the subray of $\gamma$ starting at $x$ and has acute angle with $\overline{xx'}$. In this case, the angle between $\overline{h(x)h(x')}$ and $h(\beta)$ is an obtuse angle, so $d(h(x'),h(\beta))=d(h(x'),h(x))$. By this geometric observation, we have:
\begin{equation}
d(x',\gamma)=d(x',\beta)\geq \frac{d(h(x'),h(\beta))}{1+\mu_1}-\mu_2=\frac{d(h(x'),h(x))}{1+\mu_1}-\mu_2 \geq \frac{d(x',x)}{(1+\mu_1)^2}-2\mu_2.
\end{equation}

Since $d(x',x)\leq R$ and $\mu_1=\frac{K\epsilon}{R}$, $d(x',\gamma)\geq\frac{d(x',x)}{(1+\mu_1)^2}-2\mu_2\geq d(x',x)-(3K\epsilon+2\mu_2)$. So
\begin{equation}
\begin{split}
\sin{\theta} & =\frac{\sinh{d(x',\gamma)}}{\sinh{d(x',x)}}\geq \frac{\sinh{(d(x',x)-(3K\epsilon+2\mu_2))}}{\sinh{(d(x',x))}}\\
& \geq \frac{\sinh{(l(p)-(3K\epsilon+2\mu_2))}}{\sinh{(l(p))}} \geq 1-(3K\epsilon+2\mu_2)\coth{(l(p))}.
\end{split}
\end{equation}

So $\frac{\pi}{2}-\theta\leq 2\sqrt{(3K\epsilon+2\mu_2)\coth{l(p)}}\leq \frac{1}{8p}$, and the same argument shows that $\theta'-\frac{\pi}{2}\leq \frac{1}{8p}$. So we have $|\theta'-\theta|\leq \frac{1}{4p}$.
\end{proof}

Since the second coordinate of $\Theta{(\gamma,\alpha,\overline{xx'})}=(\theta,\phi)$ is $\phi=\frac{\pi}{2}$, we need only to estimate the second coordinate $\phi'$ of $\Theta{(h(\gamma),h(\alpha),\overline{h(x)h(x')})}=(\theta',\phi')$.

\begin{prop}\label{phi}
For small enough $\epsilon>0$ and large enough $R$, we have $|\phi'-\frac{\pi}{2}|\leq \frac{1}{4p}$.
\end{prop}

We first need the following local estimation.

\begin{lem}\label{twoplanes}
If $\alpha \subset \partial \widetilde{g}(\widetilde{\mathcal{A}})$ is the common perpendicular of geodesics $\gamma_1,\gamma_2 \subset \partial \widetilde{g}(\widetilde{\mathcal{C}})$ such that $z_i=\gamma_i\cap \alpha$, and $y\in \gamma_1$ is a point on $\gamma_1$ such that $d=d(y,z_1)\leq R$. Let $P_1$ be the hyperbolic plane containing $\gamma_1$ and $z_2$, $P_2$ be the hyperbolic plane containing $\gamma_2$ and $y$, and $\psi$ be the angle between $P_1$ and $P_2$, then $\psi\leq \frac{10\epsilon}{R}$.
\end{lem}

\begin{proof}
Let $\beta=\angle{z_1z_2y}$, and $b+i\xi=\bold{d}_{\alpha}(\gamma_1,\gamma_2)$. Here we choose orientations for $\alpha, \gamma_1, \gamma_2$ such that $b>0$ and $\xi$ is close to $0$.

By the hyperbolic cosine law of right-angled hexagons,
\begin{equation}
\bold{d}_{\alpha}(\gamma_1,\gamma_2)=b+i\xi=2e^{-\frac{R}{4}+\frac{\epsilon_3-\epsilon_2-\epsilon_1}{2R}}+O(e^{-\frac{3R}{4}})+O(e^{-\frac{3R}{4}})\frac{\epsilon}{R}i,
\end{equation}
here the two $O(e^{-\frac{3R}{4}})$s are two different real functions. So $b\geq e^{-\frac{R}{4}}$, and $|\xi|\leq \frac{4\epsilon}{R}e^{-\frac{R}{4}}$.

An elementary computation gives $\cos{\psi}=\frac{\cos{\beta}\cos{\xi}}{\sqrt{1-\sin^2{\beta}\cos^2{\xi}}}$, and
\begin{equation}\label{19}
\sin^2{\psi}=\frac{\sin^2{\xi}}{1-\sin^2{\beta}\cos^2{\xi}}.
\end{equation}

For the angle $\beta$, we have the following estimation:
\small{\begin{equation}\label{20}
\sin^2{\beta}=\frac{\sinh^2{d}}{\sinh^2{d(y,z_2)}}=\frac{\cosh^2{d}-1}{\cosh^2{d}\cosh^2{b}-1}\leq \frac{\cosh^2{R}-1}{\cosh^2{R}\cosh^2{(e^{-\frac{R}{4}})}-1}=1-e^{-\frac{R}{2}}+O(e^{-R}).
\end{equation}}

By \eqref{19} and \eqref{20}
\small{\begin{equation}
\sin^2{\psi}\leq \frac{\sin^2{\xi}}{1-(1-e^{-\frac{R}{2}}+O(e^{-R}))\cos^2{\xi}}\leq \frac{\sin^2{(\frac{4\epsilon}{R}e^{-\frac{R}{4}})}}{1-(1-e^{-\frac{R}{2}}+O(e^{-R}))\cos^2{(\frac{4\epsilon}{R}e^{-\frac{R}{4}})}}\leq \frac{32\epsilon^2}{R^2}.
\end{equation}}

So $\psi\leq \frac{10\epsilon}{R}$.

\end{proof}

Now we are ready to prove Proposition \ref{phi}.
\begin{proof}
Let $x=x_0,x_1,\cdots,x_k=x'$ be the consecutive intersection points of $\overline{xx'}$ with $Y=\widetilde{\mathcal{C}}\cup \widetilde{\mathcal{A}}$, and $y_i=h(x_i)$. Let $C_i$ be the geodesic (segment) in $\partial \widetilde{g}(\widetilde{\mathcal{C}})$ or $\partial \widetilde{g}(\widetilde{\mathcal{A}})$ containing $y_i$.

Let $P_0$ be the hyperbolic plane containing $C_0=h(\gamma)$ and $h(\alpha)$, $P_0'$ be the hyperbolic plane containing $C_0$ and $y_1$. Let $P_i$ be the hyperbolic plane containing $C_i$ and $y_{i-1}$, and $P_i'$ be the hyperbolic plane containing $C_i$ and $y_{i+1}$ for $i=1,\cdots,k-1$. When considering about $\rho_0$, all the corresponding hyperbolic planes coincide with each other. So for each hyperbolic plane $P_i$ and $P_i'$, we can give it an orientation such that the corresponding orientations coincide with each other, when considering $\rho_0$.

Since $d(x,x_{k-1})\leq l(p)\leq \frac{1}{10000p^2}$, by Lemma 2.3 of \cite{KM1}, $k\leq(2+R)e^5 l(p)\leq 500Rl(p)$. The possible position of $y_{i-1},y_i$ and geodesics $C_{i-1}, C_i$ looks like a) or b) in Figure 2 for $i=1,\cdots,k-1$. By the choice of $x'$, the possible position of $y_{k-1},y_k$ and $C_{k-1},C_k$ looks like a), b) or c) in Figure 2.

\begin{center}
\psfrag{a}[]{$y_{i-1}$} \psfrag{b}[]{$y_i$} \psfrag{c}[]{$y_{k-1}$} \psfrag{d}[]{$y_k$}
\psfrag{e}[]{a)} \psfrag{f}[]{b)} \psfrag{g}[]{c)} \psfrag{h}[]{$b$} \psfrag{i}[]{$a$}
\psfrag{j}[]{$c$} \psfrag{k}[]{$d$}\psfrag{l}[]{$e$}\psfrag{m}[]{$f$}
\includegraphics[width=5in]{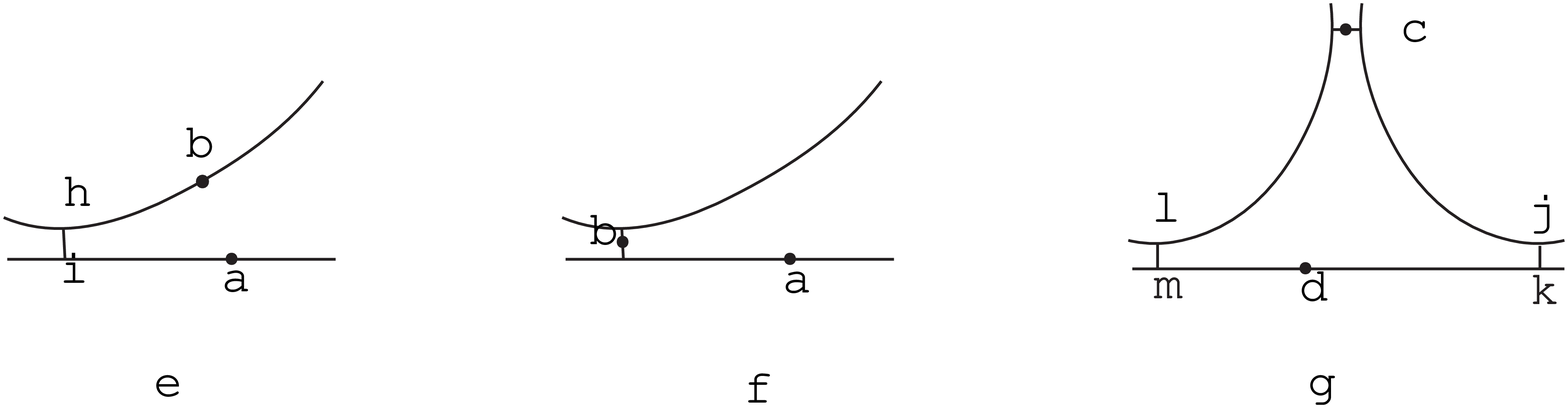}
\vskip 0.5 truecm
 \centerline{Figure 2}
\end{center}

Now we study two possible positions of $y_{i-1},y_i$ and $C_{i-1},C_i$.

i) For three points $a,b,c\in\mathbb{H}^3$ not lying on a geodesic, let $P_{abc}$ be the hyperplane in $\mathbb{H}^3$ containing these three points. Then in Figure 2 a), since $d(y_{i-1},a),d(y_i,b)<R$, Lemma \ref{twoplanes} implies:
\begin{equation}
\angle{(P_{ay_{i-1}y_i},P_{aby_{i-1}})}\leq \angle{(P_{ay_{i-1}y_i},P_{aby_i})}+\angle{(P_{aby_i},P_{aby_{i-1}})}\leq \frac{10\epsilon}{R}+\frac{10\epsilon}{R}=\frac{20\epsilon}{R},
\end{equation}
and
\begin{equation}
\angle{(P_{i-1}',P_i)}=\angle{(P_{ay_{i-1}y_i},P_{by_{i-1}y_i})}\leq \angle{(P_{ay_{i-1}y_i},P_{aby_{i-1}})}+\angle{(P_{aby_{i-1}},P_{by_{i-1}y_i})}\leq \frac{30\epsilon}{R}.
\end{equation}

ii) In Figure 2 c), let $\beta$ be the geodesic arc in $\partial \widetilde{g}(\widetilde{\mathcal{A}})$ containing $y_{k-1}$ and we use $\beta'$ to denote the bi-infinite geodesic containing $\beta$. Let $P_{\beta',c}$ be the hyperbolic plane containing $\beta'$ and $c$, and $l=d(d,P_{\beta',c})$, then
\begin{equation}
\sinh{l}\leq \sin{\frac{\epsilon}{R}}\sinh{(3e^{-\frac{R}{4}})}\leq \frac{4\epsilon}{R}e^{-\frac{R}{4}}.
\end{equation}
So $l\leq \frac{4\epsilon}{R}e^{-\frac{R}{4}}$. Since $d(d,\beta')\geq \frac{R}{2}-1$,
\begin{equation}
\sin{\angle{(P_{\beta',c},P_{\beta',d})}}\leq \frac{\sinh{(\frac{4\epsilon}{R}e^{-\frac{R}{4}})}}{\sinh{(\frac{R}{2}-1)}}\leq \frac{50\epsilon}{R}e^{-\frac{3R}{4}}.
\end{equation}
So $\angle{(P_{\beta',c},P_{\beta',d})}\leq \frac{100\epsilon}{R}e^{-\frac{3R}{4}}$. The same argument shows that $\angle{(P_{\beta',e},P_{\beta',f})}\leq \frac{100\epsilon}{R}e^{-\frac{3R}{4}}$. Since $\angle{(P_{\beta',c},P_{\beta',e})}\leq \frac{4\epsilon}{R}e^{-\frac{R}{4}}$, and by monotonicity, we have $\angle{(P_{\beta',c},P_{\beta',y_k})}\leq \frac{\epsilon}{R}$.

Then by a routine case-by-case argument, and using results in i) and ii), we get that:
\begin{equation}\label{27}
\angle{(P_{i-1}',P_i)}\leq \frac{100\epsilon}{R}
\end{equation}
for $i=1,\cdots,k-1$, and
\begin{equation}\label{28}
\angle{(P_i,P_i')}\leq \frac{100\epsilon}{R}
\end{equation}
for $i=0,\cdots,k-1$.

Now let $l_i=d(y_i,y_{i+1})$. Then by \eqref{lp},
\begin{equation}
\sum_{i=0}^{k-2} l_i \leq 2l(p)
\end{equation}
and
\begin{equation}
\frac{l(p)}{2} \leq \sum_{i=0}^{k-1} l_i \leq R+1.
\end{equation}
Let $\vec{v}_i \in T^1_{y_i}(\mathbb{H}^3)$ be the unit normal vector of $P_i$ at $y_i$ and $\vec{v}_i' \in T^1_{y_i}(\mathbb{H}^3)$ be the unit normal vector of $P_i'$ at $y_i$. Let $z_i$ be the orthogonal projection of $y_i$ to $P_0$ and $\vec{n}_i$ be the unit normal vector of $P_0$ at $z_i$. For $x,y\in \mathbb{H}^3$ and $\vec{v}\in T_x(\mathbb{H}^3)$, we will use $\vec{v}@y$ to denote the parallel transportation of $\vec{v}$ to $y$ along the geodesic arc $\overline{xy}$, as in \cite{KM1}.

{\bf Claim 1}: For $i=0,\cdots,k-1$, the following inequalities hold:
\begin{equation}
\angle{(\vec{v}_i@y_0,n_0)}\leq \frac{200i\epsilon}{R}+\sum_{j=0}^{i-1}l_j,
\end{equation}
\begin{equation}
\angle{(\vec{v}_i'@y_0,n_0)}\leq \frac{(200i+100)\epsilon}{R}+\sum_{j=0}^{i-1}l_j.
\end{equation}

We will prove Claim 1 by induction. The statement holds for $i=0$ since $\angle{(P_0,P_0')}\leq \frac{100\epsilon}{R}$. Suppose the statement holds for $i=m$, then for $i=m+1$,
\begin{equation}
\begin{split}
& \angle{(\vec{v}_{m+1}@y_0,\vec{n}_0)}\leq \angle{(\vec{v}_{m+1},\vec{v}_m'@y_{m+1})}+\angle{(\vec{v}_m'@y_{m+1}@y_0,\vec{n}_0)}\\
\leq & \angle{(P_{m+1},P_m')}+\angle{(\vec{v}_m'@y_0,\vec{n}_0)}+\angle{(\vec{v}_m'@y_{m+1}@y_0,\vec{v}_m'@y_0)}.
\end{split}
\end{equation}

The first term is less than $\frac{100\epsilon}{R}$ by \eqref{27}, the second term is less than $\frac{(200m+100)\epsilon}{R}+\sum_{j=0}^{m-1}l_j$ by induction hypothesis, and the third term is less than $l_m$ by Proposition 4.1 of \cite{KM1}. So
\begin{equation}
\angle{(\vec{v}_{m+1}@y_0,\vec{n}_0)}\leq \frac{200(m+1)\epsilon}{R}+\sum_{j=0}^ml_j.
\end{equation}

Since $\angle{(\vec{v}_{m+1},\vec{v}_{m+1}')}\leq \frac{100\epsilon}{R}$, the second inequality holds for $i=m+1$ and the proof of Claim 1 is done.

Now we estimate $\angle{(\vec{v}_m'@z_m,\vec{n}_m)}$ for $m=0,\cdots,k-1$. Since $k\leq 500Rl(p)$, we have:
\begin{equation}
\begin{split}
&\angle{(\vec{v}_m'@z_m,\vec{n}_m)}\leq \angle{(\vec{v}_m'@z_m,\vec{v}_m'@z_0@z_m)}+\angle{(\vec{v}_m'@z_0@z_m,\vec{n}_m)}\\
\leq & d(y_0,y_m)+\angle{(\vec{v}_m'@y_0,\vec{n}_0)} \leq 2l(p)+\frac{(200m+100)\epsilon}{R}+\sum_{j=0}^{m-1}l_j\\
\leq & 4l(p)+2\cdot 10^5\epsilon l(p)\leq 10 l(p).
\end{split}
\end{equation}

{\bf Claim 2}: For $d_i=d(y_i,z_i)=d(y_i,P_0),\ i=0,\cdots,k-1$, we have:
\begin{equation}
d(y_i,z_i)\leq \frac{1}{100p}\sum_{j=0}^{i-1}l_j.
\end{equation}

We will also prove Claim 2 by induction. The inequality holds trivially for $i=0$. Suppose the inequality holds for $i=m$. For $i=m+1$, Since $\vec{v}_m'\perp\overline{y_my_{m+1}}$ and $\angle{(\vec{v}_m'@z_m,\vec{n}_m)}\leq 10 l(p)$, an elementary computation gives:
\begin{equation}
\sinh{d(y_{m+1},z_{m+1})}=\sinh{d(y_m,z_m)}\cosh{l_m}+\cosh{d(y_m,z_m)}\sinh{l_m}\sin{\theta_m}
\end{equation}
for some $\theta_m$ with $|\theta_m|\leq \angle{(\vec{v}_m'@z_m,\vec{n}_m)}\leq 10 l(p)$ and $l_m\leq 2l(p)$.

Since $d(y_m,z_m)$, $l_m\leq 2l(p)$, we have:
\begin{equation}\label{38}
\begin{split}
\sinh{d(y_{m+1},z_{m+1})}& \leq \sinh{d(y_m,z_m)}(1+l_m^2)+(1+4l(p^2))\cdot2l_m\cdot 10(p)\\
& \leq \sinh{d(y_m,z_m)}+2d(y_m,z_m)l_m^2+(1+4l(p^2))\cdot2l_m\cdot 10l(p)\\
& =\sinh{d(y_m,z_m)}+l_m\left(2d(y_m,z_m)l_m+20l(p)(1+4l(p^2)\right)\\
& \leq \sinh{d(y_m,z_m)}+100l(p)\cdot l_m.
\end{split}
\end{equation}
So $d(y_{m+1},z_{m+1})\leq d(y_m,z_m)+100l(p)\cdot l_m\leq d(y_m,z_m)+\frac{1}{100p}l_m$, and the proof of Claim 2 is done.

The final computation is to estimate $\angle{z_ky_0y_k}$.

Since $h$ is a $(1+\frac{K\epsilon}{R},K(\epsilon+\frac{1}{R})^{\frac{1}{5}})$-quasi-isometry, and $\frac{l(p)}{2}\leq \sum_{i=0}^{k-1}l_i\leq R$, $d(y_0,y_k)\geq \max{(\frac{1}{2}\sum_{i=0}^{k-1}l_i,l_{k-1}-1)}$. Moreover, $d(y_k,z_k)$ is given by:
\begin{equation}
\sinh{d(y_k,z_k)} =\sinh{d(y_{k-1},z_{k-1})}\cosh{l_{k-1}}+\cosh{d(y_{k-1},z_{k-1})}\sinh{l_{k-1}}\sin{\theta_{k-1}}
\end{equation}
for some $\theta_{k-1}$ with $|\theta_{k-1}|\leq 10 l(p)$

If $l_{k-1} \leq 2$, (\ref{38}) still works with $m$ replaced by $k-1$, so $d(y_k,z_k)\leq \frac{1}{100p}\sum_{i=0}^{k-1}l_i$. Then
\begin{equation}
\sin{\angle{z_ky_0y_k}}=\frac{\sinh{d(y_k,z_k)}}{\sinh{d(y_k,y_0)}}\leq \frac{\sinh{\frac{\sum_{i=0}^{k-1}l_i}{100p}}}{\sinh{\frac{\sum_{i=0}^{k-1}l_i}{2}}}\leq \frac{1}{50p}.
\end{equation}

If $l_{k-1} \geq 2$, we have
\begin{equation}
\sinh{d(y_k,z_k)} \leq 4l(p)\cosh{l_{k-1}}+20l(p)\sinh{l_{k-1}}\leq 20l(p)e^{l_{k-1}}.
\end{equation} Then
\begin{equation}
\sin{\angle{z_ky_0y_k}}=\frac{\sinh{d(y_k,z_k)}}{\sinh{d(y_k,y_0)}}\leq \frac{20l(p)e^{l_{k-1}}}{\sinh{(l_{k-1}-1)}}\leq \frac{20l(p)\cdot e^2}{\sinh{1}}\leq \frac{1}{50p}.
\end{equation}

So in both of these cases, $|\psi'-\frac{\pi}{2}|=\angle{z_ky_0y_k}\leq \frac{1}{4p}$.
\end{proof}

\subsection{Proof of Theorem \ref{injectivity1}}

Given Theorem \ref{angle}, we are ready to prove Theorem \ref{injectivity1}.

\begin{proof}
Given the estimations in Theorem \ref{angle}, the proof here is similar to the proof of Lemma \ref{qi}, so we will only point out the necessary modifications.

In Theorem \ref{injectivity1}, if conditions 2), 3) and 4) are replaced by $\bold{hl}(C)=\frac{R}{2}$, $s(C)=1$ and $\bold{l}(q(C))=\frac{R+2k\pi i}{d_C}$ respectively, then we denote the corresponding representation by $\rho_0$. Since $l(G(X))>Re^{\frac{R}{4}}$, the same argument as in Lemma \ref{qi} shows that Theorem \ref{injectivity1} is true for $\rho_0$.

Let $q: \widetilde{X}\rightarrow X$ be the universal cover. When considering $\rho_0$, since $l(G(X))>Re^{\frac{R}{4}}$, there is a $\pi_1(X)$-equivariant embedding $\widetilde{X}\rightarrow \mathbb{H}^3$ with respect to $\rho_0$, and the image is a locally finite union of subsets of hyperbolic planes. There are two metrics on $\widetilde{X}$, one is the induced metric $d_0$ from $\mathbb{H}^3$ and the other one is the path metric $d$ induced by $d_0$. The proof of Lemma \ref{qi} implies that these two metrics are quasi-isometric, and we will always endow $\widetilde{X}$ with the metric $d$. $d$ is a geodesic metric on $X$ which is locally the hyperbolic metric away from singular curves.

Let $\widetilde{\mathcal{C}}\subset \widetilde{X}\subset \mathbb{H}^3$ be the preimage of the pants decomposition $\mathcal{C}_1\cup \mathcal{C}_2$ of $X$, which are union of geodesics. For each pants $\Pi$ in $X$, there are three seams which are the common perpendiculars of the three pair of cuffs of $\Pi$. Let $\mathcal{A}\subset X$ be the union of all such geodesic arcs and let $\widetilde{\mathcal{A}}\subset \widetilde{X}\subset \mathbb{H}^3$ be the preimage of $\mathcal{A}$ in $\mathbb{H}^3$. We define $Z=\widetilde{\mathcal{C}}\cup \widetilde{\mathcal{A}}$, then the embedding of $Z$ into $\mathbb{H}^3$ is $\pi_1(X)$-equivariant with respect to $\rho_0$.

Now we turn to study a general representation $\rho:\pi_1(X)\rightarrow PSL_2(\mathbb{C})$, which is a deformation of $\rho_0$. For each geodesic $C \subset \widetilde{\mathcal{C}}$, let $C'$ be the corresponding geodesic in $\mathbb{H}^3$ with respect to $\rho$, and let $\widetilde{\mathcal{C}'}$ be the union of such $C'$ for all $C \subset \widetilde{\mathcal{C}}$. Let $\widetilde{\mathcal{A}'}$ be the union of common perpendiculars of geodesics in $\widetilde{\mathcal{C}'}$, which correspond with geodesic arcs in $\widetilde{\mathcal{A}}$. Let $Z'=\widetilde{\mathcal{C}'}\cup \widetilde{\mathcal{A}'}$, then $Z'$ is $\rho(\pi_1(X))$-equivariant with respect to $\rho$.

There is a natural piecewise linear map $h:Z\rightarrow Z'$ defined as in Section \ref{piecewise}, such that the restriction of $h$ on each geodesic arc in $\widetilde{\mathcal{A}}$ is linear, and the restriction of $h$ on each component of $\widetilde{\mathcal{C}}\setminus \widetilde{\mathcal{A}}$ is linear. To show the convex cocompact property of $\rho$, we need only to show that $h:(Z,d)\rightarrow (Z',d_{\mathbb{H}^3}|_{Z'}$ is a quasi-isometry.

For any $x,y\in Z$ with $d(x,y)\geq R$, let $\alpha$ be the shortest path in $\widetilde{X}$ connecting $x$ and $y$. If $\alpha\cup q^{-1}(\mathcal{C}_2)=\emptyset$, the Theorem \ref{quasi} implies the quasi-isometric property. So we assume $\alpha\cap q^{-1}(\mathcal{C}_2)$ is not empty, and the intersection points are $x_1,x_2,\cdots,x_k$. Since $l(G(X))>Re^{\frac{R}{4}}$, $d(x_i,x_{i+1})\geq \frac{R}{2}$ for $i=1,\cdots,k-1$. Let $x_0=x$ if $d(x,x_1)\geq \frac{R}{2}$, or let $x_0=x_1$. Similarly, define $x_{k+1}=y$ if $d(x_k,y)\geq \frac{R}{2}$, or let $x_{k+1}=x_k$. Now we consider about angles $\angle{h(x_{i-1})h(x_i)h(x_{i+1})}$.

Let $\gamma_i$ be the geodesic in $Z$ containing $x_i$, with a preferred orientation, and let $\theta$ be the angle between $\gamma_i$ and $\overline{x_{i-1}x_i}$. Since $\alpha$ is a shortest path in $\widetilde{X}$, the angle between $\gamma_i$ and $\overline{x_ix_{i+1}}$ equals $\pi-\theta$. Without lose of generality, we can suppose $\theta \leq \frac{\pi}{2}$.

Let $\theta_1$ be the angle between $h(\gamma_i)$ and $\overline{h(x_{i-1})h(x_i)}$, and $\theta_2$ be the angle between $h(\gamma_i)$ and $\overline{h(x_i)h(x_{i+1})}$. If $\theta\leq \frac{\pi}{2}-\frac{3}{2p}$, then by Theorem \ref{angle}, we have $\theta_1\leq \frac{\pi}{2}-\frac{1}{2p}$ and $\theta_2\geq \frac{\pi}{2}+\frac{1}{2p}$, so $\angle{h(x_{i-1})h(x_i)h(x_{i+1})}\geq \frac{1}{p}$.

If $\frac{\pi}{2}-\frac{3}{2p}\leq \theta \leq \frac{\pi}{2}$, then $\frac{\pi}{2}-\frac{5}{2p}\leq \theta_1 \leq \frac{\pi}{2}+\frac{1}{p}$ and $\frac{\pi}{2}-\frac{1}{p}\leq \theta_1 \leq \frac{\pi}{2}+\frac{5}{2p}$. Let $\alpha_i$ be the component of $\widetilde{\mathcal{A}}$ which is on the same component of $\widetilde{X}\setminus q^{-1}(\mathcal{C}_2)$ as $x_{i-1}$, intersecting with $\gamma_i$, and is the closest such arc from $x_i$. We also choose $\alpha_{i}'$ by the same way with $x_{i-1}$ replaced by $x_{i+1}$. Let $P_i$ be the hyperbolic plane containing $h(\gamma_i)$ and $h(\alpha_i)$, and $P_i'$ be the hyperbolic plane containing $h(\gamma_i)$ and $h(\alpha_i')$. Since pants adjacent to the same singular curve are $p$-separated, $\angle{(P_i,P_i')}\geq \frac{2\pi}{p}$.

Let $\vec{n}_i$ and $\vec{n}_i'$ be the normal vectors of $P_i$ and $P_i'$ at $h(x_i)$ respectively, then $\angle{(\vec{n}_i,\vec{n}_i')}\geq \frac{2\pi}{p}$. Theorem \ref{angle} implies that, $|\angle{(\overline{h(x_{i-1})h(x_i)},\vec{n}_i)}-\frac{\pi}{2}|\leq \frac{1}{p}$ and  $|\angle{(\overline{h(x_i)h(x_{i+1})},\vec{n}_i')}-\frac{\pi}{2}|\leq \frac{1}{p}$. An elementary computation implies that $$\angle{h(x_{i-1})h(x_i)h(x_{i+1})}\geq \cos{\frac{1}{p}}\sin{\frac{2\pi}{p}}-\frac{2}{p}\geq (2\cos{\frac{1}{2}}-1)\frac{2}{p}>0.$$

Then the remaining part of the proof is same with the proof of Lemma \ref{qi}.
\end{proof}

\begin{rem}
Given the cited theorems, our proof of Theorem \ref{injectivity1} is very elementary and quite geometric flavor, but a little bit tedious. It is possible to give alternative proofs by using geometric group theory or using the method in \cite{Sa}.
\end{rem}

\bibliographystyle{amsalpha}

\end{document}